\documentclass[12pt,reqno]{amsart}
\usepackage{amssymb,amsmath,amsthm,newlfont,enumerate}

\usepackage{a4wide}
\usepackage{graphicx}
\usepackage{color}
\newtheorem{thm}{Theorem}[section]
\newtheorem{prop}[thm]{Proposition}
\newtheorem{cor}[thm]{Corollary}
\newtheorem{deft}[thm]{Definition}

\newtheorem{lem}[thm]{Lemma}

\newcommand{\cR}{{\mathcal{R}}}

\newcommand{\cL}{{\mathcal{L}}}

\newcommand{\cC}{{\mathcal{C}}}

\newcommand{\cW}{{\mathcal{W}}}
\newcommand{\fact}[1]{#1\mathpunct{}!}

\newcommand{\Card}{\mathrm{Card}}

\newcommand\CC{\mathbb{C}}

\newcommand\TT{\mathbb{T}}
\newcommand\DD{\mathbb{D}}
\newcommand\NN{\mathbb{N}}
\newcommand\RR{\mathbb{R}}

\begin{document}

\title[Singular values of Hankel operators]{On singular values of Hankel operators on Bergman spaces}


 \author{M. Bourass}
 \address{Mohammed V University in Rabat, Faculty of sciences, CeReMAR -LAMA- B.P. 1014 Rabat, Morocco}
 \email{marouane.bourass@gmail.com}

 \author{O. EL-Fallah}
 \address{Mohammed V University in Rabat, Faculty of sciences, CeReMAR -LAMA- B.P. 1014 Rabat, Morocco}
 \email{omar.elfallah@gmail.com; o.elfallah@um5r.ac.ma}
 
 \thanks{Research partially supported by "Hassan II Academy of Science and Technology"for the first and the second authors.}
  
 \author{I. Marrhich}
 \address{Laboratoire Mod\'elisation, Analyse, Contr\^ole et Statistiques,  Faculty of sciences Ain-Chock, Hassan II University of Casablanca, B.P 5366 Maarif, Casablanca, Morocco}
 \email{ibrahim.marrhich@univh2c.ma; brahim.marrhich@gmail.com}
 
 \author{H. Naqos}
 \address{Mohammed V University in Rabat, Faculty of sciences, CeReMAR -LAMA- B.P. 1014 Rabat, Morocco}
 \email{hatim.ns@gmail.com}


\begin{abstract}
In this paper, we study the behavior of the singular values of Hankel operators on weighted Bergman spaces $A^2_{\omega _\varphi}$, where $\omega _\varphi= e^{-\varphi}$ and $\varphi$ is a subharmonic function.
We consider compact Hankel operators  $H_{\overline {\phi}}$, with anti-analytic symbols ${\overline {\phi}}$, and give estimates of the trace of $h(|H_{\overline \phi}|)$ for any convex function $h$. This allows us to 
give asymptotic estimates of the singular values $(s_n(H_{\overline {\phi}}))_n$ in terms of decreasing rearrangement of $|\phi '|/\sqrt{\Delta \varphi}$. For the radial weights, we first prove that the critical decay of $(s_n(H_{\overline {\phi}}))_n$ is achieved  by $(s_n (H_{\overline{z}}))_n$. Namely, we establish that if $s_n(H_{\overline {\phi}})= o (s_n(H_{\overline {z}}))$, then $H_{\overline {\phi}} = 0$. Then, we show that if $\Delta \varphi (z) \asymp \frac{1}{(1-|z|^2)^{2+\beta}}$ with $\beta \geq 0$, then $s_n(H_{\overline {\phi}}) = O(s_n(H_{\overline {z}}))$ if and only if $\phi '$ belongs to the Hardy space $H^p$, where $p= \frac{2(1+\beta)}{2+\beta}$. Finally, we compute the asymptotics of $s_n(H_{\overline {\phi}})$ whenever $ \phi ' \in H^{p }$.\\

\end{abstract}
\maketitle


\section{Introduction} 
Hankel operators are one of the most important classes of bounded linear operators acting on spaces of analytic functions. They have many connections with function theory, harmonic analysis, approximation theory, moment problems, spectral theory, orthogonal polynomials, stationary Gaussian processes, $\overline{\partial}$-operator \dots etc. The book of V. Peller  \cite {Pel} is an acknowledged reference in the classical theory of Hankel operators and their various applications.
We are interested in the behavior of the singular values of Hankel operators with anti-analytic symbols acting on Bergman spaces. The first work in this subject is due to Axler \cite{Axler} who described boundedness and compactness of such operators in the classical Bergman space. Right after, Arazy, Fischer and Peetre \cite{AFP} studied  the membership for such operators in Schatten classes. They highlighted the existence of a cut-off by proving that these operators can not be of finite trace.  In \cite {ER}, Engli\v{s} and Rochberg described, using the Boutet de Monvel-Guillemin theory, all Hankel operators with anti-analytic symbols that belongs to the Dixmier class and gave an explicit formula for Dixmier Trace  in this case. Several authors considered the same problems in other spaces of analytic functions \cite{Lue2, Jan, BEY, SY, PZZ, Zhu, Pau}.  In this paper, we give the asymptotic behavior of the singular values of compact Hankel operators, we describe the class of symbols for which the associated Hankel operators have the critical decay and compute the asymptotics of  singular values of Hankel operators associated with this class of symbols.\\

The space of all holomorphic functions on the unit disc $\mathbb{D}$ in the complex plane $\mathbb{C}$ will  be denoted by $\text{Hol}(\mathbb{D})$. The Lebesgue measure on $\CC$ is denoted by $dA$.
The standard Bergman space $A^2_{\alpha}$ with $\alpha >-1$, consists of holomorphic functions $f$ on $\DD$ such that $$\| f \|^2 _\alpha :=  \displaystyle  \int _\mathbb{D} |f(z)|^2  dA_\alpha(z) <\infty,$$
where $dA_\alpha(z):= \frac{(\alpha+1)}{\pi}(1-|z|^2)^\alpha dA(z)$.\\
Recall that $A^2_{\alpha}$ is a reproducing kernel Hilbert space with the kernel
$$
K(z,w)=\frac{1}{(1-\overline{w}z)^{\alpha+2}}, \quad z,w\in\mathbb{D}.
$$
The orthogonal projection of $L^2_\alpha:=L^2(dA_\alpha)$ onto $A^2_\alpha$ will be denoted by $P_\alpha$. 
Given $g \in L^1(dA_\alpha)$, the linear transformation
\begin{equation*}
H_g f= gf-P_\alpha (gf)
\end{equation*}
is a densely defined operator of $A^2_\alpha$ into $L^2_\alpha \ominus A^2_\alpha$. The operator $H_g$ is called the (big) Hankel operator with symbol $g$.
For general facts concerning Hankel operators on Bergman spaces we refer to \cite{AFP, Zhu}.
 \noindent In this paper we are interested in Hankel operator $H_{\overline{\phi}}$ with anti-analytic symbol $\overline{\phi}$.
\noindent In \cite{AFP}, J. Arazy, S. Fisher and J. Peetre proved that $H_{\overline{\phi}}$ is bounded ($resp.$ compact) on $A^2_\alpha$ if and only if $\phi$ belongs to the Bloch space $\mathcal{B}$ ($resp.$ $\phi$ belongs to the little Bloch space $\mathcal{B}_0$). This result was first proved by S. Axler \cite{Axler} in the case $\alpha=0$.
\noindent The membership in Schatten classes of Hankel operators $H_{\overline{\phi}}$ was also studied by J. Arazy, S. Fisher and J. Peetre in \cite{AFP}. They proved that  $H_{\overline{\phi}}\in \mathcal{S}_p(A^2_\alpha)$, for $p>1$, if and only if $\phi$ belongs to the Besov space $\mathcal{B}_p$ defined by
$$
\mathcal{B}_p= \lbrace \phi \in \text{Hol}(\mathbb{D}),\quad \int_\mathbb{D} |\phi '(z)|^p(1-|z|^2)^{p-2} dA(z)<\infty\rbrace.
$$
For $p\leq 1$, they proved that $H_{\overline{\phi}}\in \mathcal{S}_p(A^2_\alpha)$ if and only if $H_{\overline{\phi}}=0$ (which means that $\phi$ is a constant). These results were extended by Galanopoulos and Pau in \cite{GP} to large Bergman spaces associated with radial weights.\\
\indent Let $\omega =: e^{-\varphi }$ be a weight on $\mathbb{D}$ such that $\varphi $ is a regular subharmonic function and let $A^2 _\omega $ be the weighted Bergman space given by 
$$
	A^2_{\omega}= \big\{ f \in \text{Hol} (\mathbb{D}):\ \ \| f \| _\omega := \left ( \displaystyle \int _\mathbb{D} |f(z)|^2dA_\omega(z) \right )^{1/2} <\infty\big\},\quad dA_\omega := \omega dA.
	$$
As before, the Hankel operator with anti-analytic symbol ${\overline {\phi}}$ is the operator $H_{\overline {\phi}}: L^2(dA_\omega )\to L^2(dA_\omega ) \ominus A^2_\omega$ given by  
\begin{equation*}
H_{\overline {\phi}} f= {\overline {\phi}}f-P_\omega ({\overline {\phi}}f),
\end{equation*}
where  $P_\omega$ is the orthogonal projection from $L^2(dA_\omega )$ onto $A^2_{\omega}$.\\

\noindent For a class of radial rapidily decreasing weights $\omega (z)= e^{-\varphi (|z|)}$, Galanopoulos and Pau proved in \cite{GP} that $H_{\overline {\phi }}$ is bounded (resp. compact) if and only if  $\displaystyle \frac{|\phi '|^2 }{\Delta \varphi}$ is bounded (resp. $\displaystyle \lim _{|z|\to 1^-}\frac{|\phi '(z)|^2 }{\Delta \varphi(z)} = 0$). They also prove, for $p>1$, that $H_{\overline {\phi }}\in S_p $  if and only if $\displaystyle \frac{|\phi '|^2 }{\Delta \varphi}  \in L^ {p/2}(\Delta \varphi dA)$.\\

Our goal in this paper is to study the asymptotic behavior of the singular values of $H_{\overline{\phi}}$. We will consider the class of weights $\cal {W}^\ast$ which covers all previous examples (see Section \ref {pre}). In order to state our main results, we introduce some notations.
The reproducing kernel of $A^2_\omega$ is denoted by $K$,
 $$
\tau _\omega (z) = \frac{1}{\omega ^{1/2}(z)\|K_z \|}\quad \mbox{ and}\quad  d\lambda _\omega = \frac{dA}{\tau _\omega ^2 }.$$
It should be noted that in several, but not all, situations $\tau _\omega ^2 $ is comparable to $1/\Delta \varphi$. For more information, see the examples given in Section \ref{pre}.\\
By analogy with the standard case, we write $\mathcal{B}^\omega $ (resp. $\mathcal{B}^\omega_0$) for the space of analytic functions $\phi$ on $\DD$ such that $ \sup_{z\in\mathbb{D} }   \tau _\omega (z)|\phi'(z)|<\infty $   
(resp. $ \displaystyle \lim _{|z|\to 1^-}   \tau _\omega (z)|\phi'(z)|=0)$.\\
The following theorem will play an important role in the sequel.
\begin{thm}\label{main result 1}
Let  $\omega \in \cal {W}^\ast$ and let $\phi \in \mathcal{B}^\omega_0$. Let $h: [0,+\infty [ \to  [0,+\infty [$ be an increasing convex  function  such that $h(0)=0$. Then there exists $B>0$, which depends only on $\omega$, such that
	$$
	\displaystyle \int _\mathbb{D} h\left ( \frac{1}{B}\tau _\omega(z)|\phi '(z)|\right ) d\lambda _\omega (z) \leq \mbox{\rm {Tr} }( h(|H_{\overline{\phi}}|) ) \leq \displaystyle \int _\mathbb{D} h\left ( B\tau_\omega (z)|\phi '(z)|\right ) d\lambda _\omega (z).
	$$
\end{thm}
Let us denote by $\mathcal {R}^+_{\phi,\omega}$ the decreasing rearrangement of the function $\tau_\omega |\phi '|$ with respect to $d\lambda _\omega$. Namely,
$${\mathcal {R}} _{\phi,\omega}^+(x) := \sup \{ t \in (0, \| \tau \phi '\| _{\infty}]:\ \cR _{\phi,\omega}(t) \geq x\}, $$
where  $\cR_{\phi,\omega}$ is the distribution function given by
$$
\mathcal{R}_{\phi ,\omega}(t) := \lambda _\omega( \{ z \in \mathbb{D}: \ \tau_\omega (z)|\phi ' (z)|  > t   \} ).
$$
As a first consequence of Theorem \ref{main result 1}, we prove that if $\rho $ is an increasing function  such that $\rho (x)/x^ \gamma$ is decreasing for some $\gamma <1$, then  
$$
 s_n(H_{\overline{\phi}} ) \asymp 1/ \rho (n) \Longleftrightarrow  \mathcal {R}^+_{\phi,\omega}(n)  \asymp 1/ \rho (n).
$$
The next result is motivated by a problem raised in  \cite {AFP, AFJP} . We prove in the following theorem that the critical decay of $(s_n(H_{\overline{\phi}}))_n$ is achieved by the symbol $\phi = z$.
 \begin{thm}\label{main result 3}
 	Let  $\omega \in \cal{W}^\ast$ be such that  $\tau_\omega $ is equivalent to a radial function and let $\phi \in \mathcal{B}^\omega_0$.  Then 
\begin{enumerate}
\item $s_n(H_{\overline{\phi}}) = o (s_n(H _{\overline{z}})) \ \Longrightarrow  
H_{\overline{\phi}}  = 0$.
\item $s_n(H_{\overline{\phi}}) = o (\mathcal {R}^+_{z,\omega}(n)) \ \Longrightarrow  
H_{\overline{\phi}}  = 0$.
\end{enumerate}
 \end{thm}
\noindent Note that, in general, it is not difficult to estimate $\mathcal {R}^+_{z,\omega}(n)$. For example, if $\tau  ^2_\omega (z) \asymp (1-|z|^2)^{2+\beta}$ with $\beta \geq 0$, then   $\mathcal {R}^+_{z,\omega}(n)\asymp n^{-\frac{2+\beta}{2(1+\beta)}}$. Thus, Theorem \ref{main result 3} implies that $$s_n(H_{\overline{\phi}}) = o \left ( n^{-\frac{2+\beta}{2(1+\beta)}} \right ) \Longrightarrow H_{\overline{\phi}}  = 0.
$$

Now our goal  is to describe the class of functions  $\phi \in \mathcal{B}^\omega_0$ such that $\left ( s_n(H_{\overline{\phi}}) \right )_n$ has the critical decay. For simplicity, we state our result only in the case $\tau^2_\omega (z) \asymp (1-|z|)^ {2+\beta}$.\\
As usual, the Hardy space  $H^p$, $p\geq 1$, consists of analytic functions $f $ on $\DD$ such that
$$\|f\| _{H^p}^p:= \sup_{0\leq r <1} \frac{1}{2\pi} \int_0^{2\pi} |f(re^{i\theta})|^p d\theta <\infty.$$

\begin{thm}\label{main result 2}
	Let  $\omega \in \cal{W}^\ast$ be such that  $\tau^2_\omega (z) \asymp (1-|z|)^ {2+\beta}$ with $\beta \geq 0$. Set $p= \frac{2(1+\beta)}{2+\beta}$ and let  $\phi \in \mathcal{B}^\omega_0$. We have 
	 $$s_n(H_{\overline{\phi}}) =O(1/n^{1/p})\iff \phi ' \in H^{p}.$$	
Furthermore, if $\omega$ is radial and  $s_n(H _{\overline{z}}) \sim \gamma /n^{\frac{1}{p}}$ for some $\gamma \in (0,\infty)$, then 
$$
s_n(H_{\overline{\phi}}) \sim \| \phi ' \| _{H^p}\frac{\gamma}{n^{\frac{1}{p}}},\quad \phi ' \in H^{p}.
$$
\end{thm} 
It should be noted that if $\tau^2_\omega (z) \asymp (1-|z|)^ {2+\beta} \log ^{-\alpha }(2/1-|z|^2)$ with $\alpha \geq 0$, then $\mathcal {R}^+_{z,\omega}(n) \asymp \frac{\log ^{\alpha/1+\beta}(n)}{n ^{1/p}}$. We prove in Theorem \ref{Hp} that if $\beta >0$, then 
$$
s_n(H_{\overline{\phi}}) =O\left ( \frac{\log ^{\alpha/1+\beta}(n)}{n ^{1/p}} \right )\iff \phi ' \in H^{p}.
$$
However, if $\beta = 0$ and $\alpha >0$, curiously, the previous result is not valid. This is the subject of Proposition \ref{CE}.\\

For the standard Bergman spaces $A^2_\alpha$ we have  $\tau^2_{\omega _\alpha} (z) \asymp (1-|z|)^ {2}$. It is known, and can be easily seen, that $s_n(H _{\overline{z}})\sim \frac{\sqrt{1+\alpha}}{n+1}$. Hence, Theorem \ref{main result 2} says that if $\phi \in \cal {B}_0$ then 
 $$s_n(H_{\overline{\phi}}) = O (1/n) \iff \phi ' \in H^1.$$
 And in this case $s_n(H_{\overline{\phi}}) \sim \frac{\sqrt{1+\alpha}}{n+1}\| \phi' \| _{H^1}$.  This last result improves the results obtained in \cite{Dos, ER}. Another example is given in Section \ref {Example}.\\
 
The paper is organized as follows: 
In Section 2 we  introduce all definitions and notations that are used in the rest of the paper.  In section 3, we give a description of boundedness and compactness of Hankel operators $H_{\overline{\phi}}$ on $A^2_\omega$. We establish, in Section 4, an upper and a lower estimates of the Trace of $h( |H_{\overline{\phi}}|)$. The upper estimate is obtained from  H\"{o}rmander type $L^2_\omega$ estimates for $\bar \partial$- equation and from recent estimates obtained by El-Fallah and El Ibbaoui for Toeplitz operators \cite {EE0, EE}. The lower estimate is obtained through a sort of local Berezin transform of $H_{\overline{\phi}}$. Two direct  consequences of trace estimates are obtained by a suitable choice of the convex function $h$. The first one gives a sharp asymptotic estimates of the singular values of  compact operators of $H_{\overline{\phi}}$. The second, presented in Section 5, shows that the critical decay of the sequence $(s_n(H_{\overline{\phi}}))_n$ is achieved by the symbol $\phi = z$. In Section 6, we prove the first assertion of Theorem \ref{main result 2}. The proof is based on Theorem \ref{main result 1} and on estimates of some maximal non-tangential functions. The second part of Theorem \ref{main result 2} is given in Section 7. The proof of this result is based on the first part of Theorem \ref{main result 2} and on a result on asymptotic orthogonality due to A. Pushnitski (see Appendix). Section 8 is devoted to an explicit example.\\
	
Throughout the paper, the notation $A \lesssim B$ means that there is a constant $c$ independent of the relevant variables such that $A \leq c B$. We write $A \asymp B$ if both $A \lesssim B$ and $B \lesssim A$. As usual, the notation $u_n \sim v_n$ means that $\displaystyle\lim_{n\to\infty}u_n/v_n=1$. The letter $C$ will denote an absolute constant whose value may change at different occurrences.\\

\noindent {\bf Acknowledgements}. The authors are grateful to Evgeni Abakumov for several helpful discussions and suggestions.
	
\section{Preliminaries}
\subsection{The class of weights $\mathcal{W}^\ast$}\label{pre}
	
	Throughout this paper, $\omega$ will denote a function from $\DD$ into $]0, \infty[$ which is integrable  with respect to the Lebesgue measure. The associated Bergman space will be denoted by $A^2_{\omega}$. We will also assume that $\omega$ is bounded below by a positive constant on each compact set of $\DD$. This implies that $A^2_\omega $ is a reproducing kernel Hilbert space. The kernel of $A^2_\omega $ will be denoted by $K$.\\
	The orthogonal projection from $L^2_\omega:=L^2(\DD, dA_\omega)$ onto $A^2_\omega$ will be designated by $P_\omega$. It can be represented as follows
$$
P_\omega (f) (z)= \displaystyle \int _\DD f(\zeta)K(z,\zeta)dA_\omega(\zeta),\quad f\in L^2_\omega.
$$	
So, the domain of $P_\omega$ can be extended to all functions $f$ such that $f K_z\in L^1_\omega:=L^1(\DD, dA_\omega)$, for all $z\in\DD$. 
Let $g\in  L^2_\omega$ such that $g K_z \in L^2_\omega$, for all $z\in\DD$. The (big) Hankel operator $H_g$ with symbol $g$ is the densely defined operator on $A^2_\omega$ defined by
\begin{equation*}
	H_g f= g f-P_\omega (g f),
	\end{equation*}
where $f = \displaystyle \sum _{1\leq i \leq n}c_i K_{z_i}$, with $c_i \in \CC$ and $z_i \in \DD$.\\
The explicit formula for $H_g$ is
	\begin{equation*}
	H_g f(z)= \int_\DD (g(z)-g(w)) f(w)K(z,w)dA_\omega(w),\quad z\in\DD.
	\end{equation*}
	We are interested in this paper in anti-analytic symbols on $\DD$, $g = \overline{\phi}$. In this case a direct computation gives the following useful  formula
	\begin{equation}\label{hankel-kernel}
	(H_{\overline{\phi}}  K_a)(z)= ({\overline{\phi}}(z)-{\overline{\phi}}(a))K_a(z),\quad z,a\in\DD.
	\end{equation}

First, we recall the definition of the class of  weights $\mathcal{W}$ introduced in \cite{EMMN}.
Let
	$$
	\tau_\omega(z)= \frac{1}{\omega^{1/2}(z)\,\|K_z\|_\omega}, \quad z\in\DD.
	$$
	Suppose that the reproducing kernel $K$ satisfies the following conditions
	\begin{equation}\label {C1}
	\displaystyle \lim _{|z|\to 1^- }\| K_z\| = \infty \  \text{and}\ \mbox{for every}  \ \zeta \in \DD , \  \ |K(\zeta, z)| = o(\|K_z\|), \qquad |z|\to 1^-.
	\end{equation}
	We will suppose that $\tau _\omega $ is such that 
	\begin{align}\label {C5}
	\tau_\omega(z) = O \left ( 1-|z| \right ),\quad
	z\in\DD,
	\end{align}
	and that there exists constant $\eta >0$ such that for $z,\zeta \in \DD$ satisfying $| z-\zeta| \leq \eta \tau_\omega (z)$, we have
	\begin{align}\label {C4}
	\tau_\omega(z) \asymp \tau_\omega (\zeta)\ \mbox{and}\  \|K _z \| _\omega \|K _\zeta \| _\omega \lesssim |K(\zeta, z)|.
	\end{align}
		If the weight $\omega$ satisfies all the previous conditions, we shall say that the weight $\omega$ belongs to the class $\mathcal{W}$. Note that  $\mathcal{W}$ contains all standard weights. For more information, see the examples listed in \cite{EMMN}.\\

The Laplace operator $\Delta$ is given by $\Delta = \partial \bar{\partial}$, with
	$$\partial := \frac{1}{2}(\frac{\partial }{\partial x}-i\frac{\partial }{\partial y}),\quad  \overline{\partial }:=   \frac{1}{2}(\frac{\partial }{\partial x}+i\frac{\partial }{\partial y}).$$
	
	\noindent Let $\omega= e^{- \varphi} \in\mathcal{W}$ such that $\varphi\in\mathcal{C}^2(\DD)$. We shall suppose that
	\begin{equation}\label {C9}
	\tau_\omega ^2(z) \Delta\varphi(z)\gtrsim 1, \quad z\in\DD,
	\end{equation}
	or, there exist a subharmonic function $\psi: \DD \to \mathbb{R}^+$ and constants $\delta>0$ and $t\in (-1,0)$ such that for all $z\in \DD$ we have
	\begin{equation}\label {C6}
	\tau_\omega ^2(z) \Delta\psi(z)\geq \delta,\ \Delta\varphi(z)\geq t \Delta \psi (z)\ \mbox{and}\ |\partial\psi(z)|^2\leq  \Delta \psi (z).
	\end{equation}
		\begin{deft}
		We say that $\omega\in\mathcal{W}^\ast$ if $\omega\in\mathcal{W}$ and satisfies (\ref{C9}) or (\ref{C6}). 
	\end{deft}
\subsubsection{Examples}
In this subsection, we give three examples which will be our references throughout this paper.
\begin{itemize}
\item Standard Bergman spaces $A^2_{\alpha}$: These spaces are associated  with $\omega _\alpha(z) = \frac{(\alpha +1)}{\pi}(1-|z|^2)^\alpha$, with $\alpha>-1$. 
The reproducing kernel of $A^2_{\alpha}$ is $
K(z,w)=\frac{1}{(1-\overline{w}z)^{\alpha+2}}.
$
Then $$ \tau _{\omega _\alpha }(z) =\sqrt{ \frac{\pi}{\alpha +1}}(1-|z|^2).$$
Clearly, $\omega _\alpha \in  \mathcal{W}$. Note that if $\alpha > 0$, then $ \tau _{\omega _\alpha}(z) \asymp 1/\sqrt{\Delta \varphi}$, where $\varphi = \log 1/\omega _\alpha$. Then $\omega _\alpha$ satisfies (\ref {C9}) and $\omega _\alpha \in \mathcal{W}^\ast$. It is also clear that if $\alpha \in (-1,0]$, then $\omega _{\alpha} $ satisfies (\ref {C6}) with $t = \alpha $ and $\psi (z) = \log (1/1-|z|^2)$. Then $\omega _\alpha \in \mathcal{W}^\ast$.
\item  Harmonically weighted Bergman spaces: In this case we suppose that $\omega $ is a positive harmonic weight. It is proved in \cite {EMMN2}, that $\tau _\omega (z) \asymp 1-|z|^2$ and $\omega \in \mathcal{W}$.  One can easily see that $\varphi  = \log (1/\omega )$ is subharmonic and that $\omega $ satisfies   (\ref {C6}), with $t\in (-1,0)$ and  $\psi (z) = \log (1/1-|z|^2)$. It should be noted that in general $\tau^2 _\omega $ and $1/\Delta  \varphi$ are not comparable.
\item  Large Bergman spaces: The following class of weights was introduced by Hu, Lv and Schuster in \cite {HLS}. It includes the classes considered in \cite{BDK, LR, MO}. Let $\mathcal {L}_0$ be the class of  functions $\tau \in  \mbox{lip} ({\mathbb D}, {\mathbb R}) $ such that $\displaystyle \lim _{|z|\to 1^-}\tau (z) =0$. We denote $\mathcal {W} _0$ the set of weights $\omega = e^{-\varphi} $, where $\varphi \in \cal {C}^2$ is strictly subharmonic and for which there exists $\tau \in \mathcal {L}_0$ such that $\tau \asymp 1/\sqrt{\Delta \varphi}$. One can directly see, from \cite {HLS}, that if $\omega = e^{-\varphi} \in \mathcal {W} _0$, then $\omega = e^{-\varphi} \in \mathcal {W} ^\ast$ and $\tau _\omega \asymp 1/\sqrt{\Delta \varphi}$.
\end{itemize}
\subsection{Some inequalities involving convex functions.}
The following elementary lemma is proved in \cite {EE}.
\begin{lem}\label {Convex1}
Let $(a_n)_{n\geq 1}, (b_n)_{n\geq 1}$ be two decreasing sequences such that $\displaystyle \lim _{n\to \infty} a_n =\displaystyle \lim _{n\to \infty} b_n = 0$ and suppose that there exists $\gamma \in (0,1)$ such that $(n^\gamma b_n)$ is increasing.
 Suppose that there exists $ B >0$ such that 
$$
\displaystyle \sum _{n\geq 1} h\left (b_n / B\right ) \leq  \displaystyle \sum _{n\geq 1} h(a_n) \leq \displaystyle \sum _{n\geq 1} h (Bb_n),
$$
for all increasing convex function $h$. Then $a_n \asymp b_n$.
\end{lem}\label {Convex2}
We also need the following elementary result. For completness, we include the proof.
\begin{lem}\label {Convex3}
Let $A, p $ be two real numbers such that $0<A< p$. Let $\rho $ be an increasing function such that $\rho (x) / x^A$ is decreasing and let $(a_n)_n$ be a decreasing sequence such that, 
$$
 \displaystyle \sum _{n\geq 1} h(a_n) \leq \displaystyle \sum _{n\geq 1} h (1/\rho (n)),
$$
for all increasing function $h$ such that $h(t^{p})$ is convex. Then
$$
a_n \leq C(p,A)/ \rho (n).
$$ 
\end{lem}
\begin{proof}
Let $\delta >0$ and consider $h(t)= \left ( t^{1/p}-\delta ^{1/p}\right )^+$. By hypothesis, 
$$
 \displaystyle \sum _{n\geq 1} \left (a_n^{1/p}-\delta ^{1/p}\right )^+ \leq \displaystyle \sum _{n\geq 1} \left (1/\rho ^{1/p}(n)-\delta ^{1/p}\right )^+.
$$
Then we have
$$
c(p) \displaystyle \sum _{a_n \geq 2\delta} a_n^{1/p} \leq \displaystyle \sum _{\rho (n) \leq 1/\delta }1/\rho ^{1/p}(n).
$$
Using the fact that $n^{A/p}/\rho ^{1/p}(n)$ is increasing and the fact that $A/p <1$, we obtain
$$
\delta ^{1/p}\Card \{n:\ a_n\geq 2\delta \}\leq C(p,A)\delta ^{1/p}\Card \{ n : \ \rho (n)\leq 1/\delta\}.
$$
This implies the result.
\end{proof}
The following elementary lemma will be useful in the proof of Theorem \ref{main result 1}.
\begin{lem}\label{0}
Let $H,\ K$ be two Hilbert spaces and let $T: H\to K$ be a compact operator. Suppose there exist $(u_n)_{n\geq 1} \subset H$, $(v_n)_{n\geq 1}\subset K$ such that $\|u_n\| = \|v_n\| =1$ and 
$$
\displaystyle \sum _n |\langle u_n, u\rangle |^2\leq C\|u\|^2 \ \mbox {and }\ \displaystyle \sum _n |\langle v_n, v\rangle |^2\leq C\|v\|^2\quad (u\in H, v\in K),
$$
for some $C>0$. Then for any increasing convex function $h$ such that $h(0)=0$, we have 
$$
\displaystyle \sum _n h( |\langle Tu_n, v_n\rangle | )\leq C \displaystyle \sum _n h(s_n(T)).
$$
\end{lem}	
\begin{proof}
It suffices to use the spectral decomposition of $T$.
\end{proof}
	\section{Boundedness and compactness of Hankel operators on $A^2_\omega$}
	We need the following  $L^2$-estimates of solutions of the $\overline{\partial}$-equation due to B. Berndtsson (\cite{Bern1}, Theorem 3.1). 
\begin{thm} [\textbf{B. Berndtsson}] \label {d-bar1}
	Let $\Omega$ be a domain in $\mathbb{C}$ and let $\varphi$ and $\psi$ be subharmonic functions in $\Omega$. Assume that $\psi$ satisfies
	\begin{equation}\label {ast}
	|\partial\psi(z)|^2\leq  \Delta \psi (z) , \qquad z\in\Omega.
\end{equation}
	Let $s \in(0,1)$. Then, for any function $g$ on $\Omega$, there exists a solution $u$ to the equation $\overline{\partial}u=g$ such that
	$$
	\int_{\Omega} |u|^2 e^{-\varphi+s\psi} dA \lesssim \int_{\Omega}\frac{|g|^2}{\Delta\psi}  e^{-\varphi+s\psi} dA.
	$$
\end{thm}

\begin{lem}\label{dbar}
Let $\omega \in \mathcal{W}^\ast$. There exists $C>0$ such that for any function $g$ on $\DD$, there exists a solution $u$ to the equation $\overline{\partial}u=g$ such that
	$$
	\int_{\DD} |u(z)|^2 \omega(z) dA (z) \leq C \int_{\DD}\tau ^2  (z) |g(z)|^2 \omega(z) dA(z).
	$$
\end{lem}	
\begin{proof}
Note that if $\omega $ satisfies (\ref{C9}), then the result comes directly from $\overline{\partial}-$H\"{o}rmander's theorem. On the other hand, suppose that $\omega $ satisfies condition (\ref{C6}), that is there exist a subharmonic function $\psi: \DD \to \mathbb{R}^+$ and constant  $t\in (-1,0)$ such that for all $z\in \DD$ we have
	\begin{equation*}
	\frac{1}{\Delta\psi(z)}\lesssim \tau_\omega ^2(z),\ t \Delta \psi (z) \leq \Delta\varphi(z) \ \mbox{and}\ |\partial\psi(z)|^2\leq  \Delta \psi (z).
	\end{equation*}
Then by Theorem \ref{d-bar1}, applied to the couple of subharmonic functions $(\varphi  - t \psi, \psi)$ with $s = - t$, there exists a solution $u$ to the equation $\overline{\partial}u=g$ such that 
	$$
	\int_{\DD} |u|^2 e^{-\varphi} dA \lesssim \int_{\DD}\frac{|g|^2}{\Delta\psi}  e^{-\varphi}dA.
	$$
 Then, we get
	\begin{equation*}
	\int_\DD |u(z)|^2\omega(z)dA(z) \lesssim \int_\DD  |g(z)|^2 \tau^2_\omega(z)\omega(z)dA(z).
	\end{equation*}
\end{proof}

In the sequel $k_z= \frac{K_z}{\| K_z\|}$ will denote the normalized reproducing kernel of  $A^2_\omega$. The following result describes the boundedness of $H_{\overline{\phi}}$ on $A^2_\omega$ when $\omega\in \mathcal{W}^\ast$.

\begin{thm}\label{bdd}
	Let $\omega \in \mathcal{W}^\ast$. Then, the Hankel operator $H_{\overline{\phi}}$ is bounded on $A^2_\omega$
	if and only if $\phi\in \mathcal{B}^\omega$. In this case $\| H_{\overline{\phi}} \| \asymp \displaystyle \sup _{z\in \DD}\tau _\omega (z) |\phi ' (z)| $, where the implied constants depend only on $\omega$.
\end{thm}
\begin{proof}
	Suppose that $H_{\overline{\phi}}$ is bounded on $A^2_\omega$. Fix $a$ in $\DD$ and let $\delta \leq \eta$ ($\eta$ is the constant which appears in (\ref {C4})). By the  formula (\ref{hankel-kernel}), we have
	\begin{align*}
	\int_\DD |H_{\overline{\phi}} k_a(z)|^2 \omega(z) dA(z) &= \int_\DD |\phi(z)-\phi(a)|^2 |k_a(z)|^2 \omega(z) dA(z)\\
	&\geq \int_{D(a,\delta\tau_\omega(a))} |\phi(z)-\phi(a)|^2 |k_a(z)|^2 \omega(z) dA(z).
	\end{align*} 
	Using the fact that $|k_a(z)|^2\asymp K(z,z)$ and the fact that $\tau_\omega(z)\asymp \tau_\omega(a)$ when $z\in D(a,\delta\tau_\omega(a))$, we have
	\begin{equation*}
	\int_{D(a,\delta\tau_\omega(a))} |\phi(z)-\phi(a)|^2 \Vert k_z\Vert^2 \omega(z) dA(z) \asymp \frac{1}{\tau^2_\omega(a)}\int_{D(a,\delta\tau_\omega(a))} |\phi(z)-\phi(a)|^2 dA(z).
	\end{equation*}
Then we obtain
$$
\int_\DD |H_{\overline{\phi}} k_a(z)|^2 \omega(z) dA(z) \gtrsim \frac{1}{\tau^2_\omega(a)}\int_{D(a,\delta\tau_\omega(a))} |\phi(z)-\phi(a)|^2 dA(z).
$$
By Cauchy's representation formula, we get
	$$
	\tau_\omega(a) |\phi'(a)|\lesssim \frac{1}{\tau^2_\omega(a)}\int_{D(a,\delta\tau_\omega(a))} |\phi(z)-\phi(a)| dA(z).
	$$
	It follows that 
	\begin{align*}
	\tau^2_\omega(a) |\phi'(a)|^2 &\lesssim \frac{1}{\tau^2_\omega(a)}\int_{D(a,\delta\tau_\omega(a))} |\phi(z)-\phi(a)|^2 dA(z)\\
	& \lesssim \int_\DD |H_{\overline{\phi}} k_a(z)|^2 \omega(z) dA(z)\\
	&= \Vert H_{\overline{\phi}} k_a\Vert^2\\
	&\leq \Vert H_{\overline{\phi}}\Vert^2.
	\end{align*}
	Hence, $\phi\in \mathcal{B}^\omega$.
	
	Suppose now that $\sup_{z\in\DD} \tau_\omega(z) |\phi'(z)|<\infty$ and let $f\in A^2_\omega$. By Lemma \ref{dbar}, there exists a  solution $u$ to the equation
	\begin{equation}\label{d-bar} 
	\overline{\partial}u= \overline{\phi'} f 
	\end{equation} 
	satisfying
	\begin{equation*}
	\int_\DD |u(z)|^2\omega(z)dA(z) \lesssim \int_\DD \tau^2_\omega(z) |\phi'(z)|^2 |f(z)|^2 \omega(z)dA(z).
	\end{equation*}
	Therefore, since $H_{\overline{\phi}}f$ is the $L^2_\omega-$minimal solution to the equation $(\ref{d-bar})$, we have
	\begin{equation}\label{d-bar2}
	\int_\DD |H_{\overline{\phi}}f(z)|^2\omega(z)dA(z) \lesssim \int_\DD \tau^2_\omega(z) |\phi'(z)|^2 |f(z)|^2 \omega(z)dA(z).
	\end{equation}
	Hence
	$$
	\Vert H_{\overline{\phi}}\Vert \lesssim \sup_{z\in\Omega} \tau_\omega(z) |\phi'(z)|<\infty.
	$$
\end{proof}  
Note that from equation (\ref {d-bar2}), if $(f_n)_n$ is an orthonormal basis of $A^2_\omega$ then 

\[
\begin{array}{lll}
\displaystyle \sum _n\| H_{\overline{\phi}}f_n\|^2& \lesssim  &\displaystyle \sum _n \displaystyle \int _ {\mathbb {D}} |f_n (z)|^2\tau ^2_\omega(z)|\phi ' (z)|^2\omega (z)dA(z)\\
& \asymp&\displaystyle \int _ {\mathbb {D}}\| K_z \| ^2\tau ^2_\omega(z)|\phi ' (z)|^2\omega (z)dA(z)\\
& = &\displaystyle \int _ {\mathbb {D}} |\phi ' (z)|^2dA(z).
\end{array}
\]
Consequently, if $\phi \in \cal {B}_2$ then $H_{\overline{\phi}} \in \cal {S}_2$. We will see in the next paragraph that the converse is also true.\\
We have the following description of the compactness of Hankel operators.
\begin{thm}
	Let $\omega \in \mathcal{W}^\ast$. Then, the Hankel operator $H_{\overline{\phi}}$ is compact on $A^2_\omega$
	if and only if $\phi\in \mathcal{B}_0^\omega$.
\end{thm}
\begin{proof}
Suppose that $H_{\overline{\phi}}$ is compact. Since $\omega \in \mathcal{W}^\ast$, $k_z = \frac{K_z}{\| K_z\|}$ converges weakly to $0$ when $|z|\to 1^-$. Then $\displaystyle \lim _{|z| \to 1^-}\| H_{\overline{\phi}}k_z\|= 0$. By the proof of Theorem \ref{bdd}, we have 
$$  \tau  (z) |\phi ' (z)| \lesssim \| H_{\overline{\phi}}k_z\| .$$
This implies that $\phi\in \mathcal{B}_0^\omega$.\\
Conversely, let $\phi \in \cal {B}^\omega _0$ and let $\phi _r (z)= \phi (rz)$. Clearly, $\phi _r$ converges to $ \phi$  in ${\cal {B}^\omega} $ as $r\to 1^-$. Then by Theorem \ref {bdd} $H_{\overline \phi _r}$ converges to $H_{\overline \phi}$. Now since $\phi _r \in {\cal {B}_2}$, $H_{\overline \phi _r}$ is compact. Hence, $H_{\overline \phi }$ is compact.
 \end{proof}

	\section{Trace estimates for Hankel operators.}
This section is devoted to the proof of Theorem \ref{main result 1}. Before starting the proof, we recall the following covering lemma.
	\begin{lem}(\cite{EMMN})\label {cov}
		Let $X$ be a subset of $\mathbb{C}$ and let $\tau : X \to (0,\infty)$ be a bounded function. Suppose that there are  two constants $\gamma, C > 0$ such that, for every $z,w \in X$ with $ | z-w | < \gamma \tau (z)$, we have 
		$\frac{\tau (z)}{C} \leq \tau (w) \leq C\tau (z)$.
		Set $B= C+1$ and let $\delta \leq \gamma/B$. There exists a sequence $(z_n)_{n \geq 1}\subset X$  such that 
		\begin{enumerate}
			\item $X \subset \displaystyle \cup _{n\geq 1} D(z_n,\delta \tau (z_n))$.
			\item $D(z_n, \delta \tau (z_n)/2C)\cap D(z_m, \delta \tau (z_m)/2C) = \emptyset$ for  $n\neq m$.
			\item For $z\in D(z_n, \delta \tau (z_n))$ we have  $D(z, \delta \tau (z)) \subset D(z_n, B\delta \tau (z_n))$.
			\item $(D(z_n, B\delta \tau (z_n)))_n$ is a covering of $X$ of finite multiplicity.
		\end{enumerate} 
		Such sequences will be  called a $(\tau ,\delta)-$ lattice of $X$. 
	\end{lem}
	We say that $(R_n)_n \in  \cL _\omega $ if  $(R_n)_n= ( D (z_n, \delta \tau _\omega (z_n)))_n$ and satisfies the conditions $(1)-(4)$ of the above lemma. In the sequel we fix  $(R_n)_n= ( D (z_n, \delta \tau _\omega (z_n)))_n \in \cL _\omega$ and $b >1$ such that  $(bR_n)_n=: ( D (z_n, b\delta \tau _\omega (z_n)))_n$ is a covering of $\DD$ of finite multiplicity. \\
		For a positive Borel measure $\mu$ on $\DD$, the Toeplitz operator associated with $\mu $ and defined on $A^2_\omega $ is given by 
	$$
	T_\mu f(z) = \displaystyle \int_\DD f(w) K(z,w)\omega (w)d\mu (w), \quad z\in \DD.
	$$
It is easy to verify that $T_\mu$ satisfies the following remarkable formula
$$
\langle T_\mu f , f\rangle = \displaystyle \int_\DD |f(z)|^2\omega (z)d\mu (z).
$$	
For more properties of Toeplitz operators we refer to \cite {Lue, EE, EE0, EMMN}.\\
The following lemma will be used in the proof of the lower estimate of $\text{Tr}\ h(|  H_{\overline{\phi}}|)$. 
\begin{lem}\label {1}
Let $\omega \in \cW ^\ast$ and let $\phi \in  \mathcal{B}^\omega$. For any increasing convex function $h$ such that $h(0)=0$, we have
$$
\displaystyle \sum _{n}h(\|  \chi _{bR_n}H_{\overline{\phi}}k_{z_n} \|) \lesssim \displaystyle \sum _nh(s_n(H_{\overline{\phi}} )),
$$
where the implied constant depends only on $(bR_n)_n$ and $\omega$.
\end{lem}
\begin{proof}
Let $u_n= k_{z_n}$, $v_n= \frac{\chi _{bR_n}H_{\overline{\phi}}k_{z_n}}{\|  \chi _{bR_n}H_{\overline{\phi}}k_{z_n} \|}$ and remark that $\|  \chi _{bR_n}H_{\overline{\phi}}k_{z_n} \| =\langle H_{\overline{\phi}}  u_n, v_n\rangle $. So, it suffices to prove that the conditions of Lemma \ref{0} are satisfied. Indeed, let $f \in A^2_\omega$ and let $d\mu = \displaystyle \sum _n \frac{1}{\| K_{z_n}\|^2\omega (z_n)}d\delta _{z_n}$. Since $\omega \in \cW$,  $T_\mu$ is bounded \cite {EMMN}. Then we have 
$$
\displaystyle \sum_n |\langle u_n,f\rangle |^2 = \displaystyle \sum _n \frac{|f(z_n)|^2}{\| K_{z_n}\|^2}=\langle T_\mu f,f\rangle  \leq \|T_\mu\| \| f\| ^2.
$$ 
Now let $g\in L^2(\omega dA)$. By Holder inequality, we have 
$$
\displaystyle \sum_n |\langle v_n,g\rangle |^2 = \displaystyle \sum_n |\langle \frac{\chi _{bR_n}H_{\overline{\phi}}k_{z_n}}{\|  \chi _{bR_n}H_{\overline{\phi}}k_{z_n} \|}, \chi _{bR_n}g\rangle |^2 \leq  \displaystyle \sum_n  \| \chi _{bR_n}g \| ^2 \lesssim \|g \|^2.
$$
This completes the proof.
\end{proof}
\begin{proof}[Proof of Theorem \ref{main result 1}]
First, we prove that $ \displaystyle \sum _{n}h\left (s _n(H_{\overline{\phi}} ) \right ) \leq \displaystyle \int _\DD h\left ( C|\phi '(z)|\tau_\omega(z)\right ) d\lambda_\omega (z)$. The equation $(\ref{d-bar2})$ implies that 
	\begin{equation} \label{hankel toeplitz}
	H_{\overline{\phi}}^\ast H_{\overline{\phi}}\lesssim T_{\mu_\phi}.
	\end{equation}
	Then, by the monotonicity Weyl's lemma, we have
\begin{equation} \label{HT}
	s_n^2(H_{\overline{\phi}})=\lambda_{n}(H_{\overline{\phi}}^\ast H_{\overline{\phi}}) \lesssim   \lambda _n(T_{\mu _\phi}).
	\end{equation}
Let $\tilde{h}(t) = h(\sqrt{t})$. 
By Theorem 4.5 of \cite{EE}, we have
	\begin{align*}
	\sum h(s_n(H_{\overline{\phi}})) & = \sum \tilde{h}(s_n^2(H_{\overline{\phi}}))\\
	&\leq \sum \tilde{h}( C\lambda _n(T_{\mu _\phi}))\\
	&\leq \sum  \tilde{h}(C \frac{\mu _\phi(R_n)}{A(R_n)})\\
	&=  \sum  \tilde{h} \left( C \frac{1}{A(R_n)}\int_{R_n} \tau_\omega^2(z) |\phi '(z)|^2 dA(z)\right) \\
	& \leq \sum \tilde{h}\left( C \int_{R_n}  |\phi '(z)|^2 dA(z)\right).
	\end{align*}
\noindent By subharmonicity, we have for all $z\in R_n$
	\begin{align*}
	|\phi '(z)|
	&\lesssim \frac{1}{A(R_n)}\int_{bR_n}  |\phi '(\zeta)| dA(\zeta).
	\end{align*}
\noindent Then,
$$
\int_{R_n}  |\phi '(z)|^2 dA(z) \lesssim \frac{1}{ A(R_n)} \left ( \int_{bR_n}  |\phi '(\zeta)| dA(\zeta)\right )^2\asymp  \left ( \int_{bR_n}  |\phi '(\zeta)| \tau _\omega (\zeta)d\lambda_\omega (\zeta)\right )^2.
$$
 Since $h$ is convex we obtain
	\begin{align*}
 \tilde{h}\left( C \int_{R_n}  |\phi '(z)|^2 dA(z)\right) &\leq  \tilde{h}\left( C   \left ( \int_{bR_n}  |\phi '(\zeta)|  \tau _\omega (\zeta)d\lambda_\omega(\zeta)\right )^2 \right) \\
 & \leq   h \left( C   \int_{bR_n}  |\phi '(\zeta)|  \tau _\omega (\zeta)d\lambda_\omega(\zeta)\right ) \\
 & \leq    \int_{bR_n}  h  \left (C |\phi '(\zeta)|\tau _\omega (\zeta) \right ) d\lambda_\omega(\zeta).
	\end{align*}
Combining these inequalities and the fact that $(bR_n)_n$ is of finite multiplicity, we get
$$
\sum h(s_n(H_{\overline{\phi}})) \leq  \int _\DD h\left( C|\phi '(z)| \tau_\omega(z)\right) d\lambda_\omega (z).
$$
	
\noindent Now we prove the lower inequality by using Lemma \ref{1}. We have 
\[
\begin{array}{lll}
\|  \chi _{bR_n}H_{\overline{\phi}}k_{z_n} \|&\asymp&\left ( \displaystyle \int _{bR_n} |\phi (z)-\phi (z_n)|^2\|K_z\|^2\omega (z)dA(z)\right )^{1/2}\\
&\asymp&\left ( \displaystyle \int _{bR_n} |\phi (z)-\phi (z_n)|^2d\lambda _\omega(z)\right )^{1/2}.\\
\end{array}
\]
By Cauchy's formula we have 
	$$
	\tau_\omega(z)|\phi '(z)| \lesssim \left ( \displaystyle \int _{bR_n}|\phi (\zeta)-\phi (z_n)|^2 d\lambda_\omega(\zeta)\right ) ^{1/2}\asymp \|  \chi _{bR_n}H_{\overline{\phi}}k_{z_n} \|, \quad z \in R_n.
	$$
Then
$$
\displaystyle \int _\DD h(\tau _\omega (z) |\phi '(z)|d\lambda _\omega (z)  \asymp  \displaystyle \sum _n \displaystyle \int _{R_n}h(\tau _\omega (z) |\phi '(z)|d\lambda _\omega (z)\leq  \displaystyle \sum _n h(C  \|  \chi _{bR_n}H_{\overline{\phi}}k_{z_n} \| ).$$
By Lemma \ref {1}, we obtain the desired result.
\end{proof}
\noindent Let us denote by 
$$
{\cal{B}}^{\omega,p}=\{ \phi \in {\cal{B}}^\omega:\  \int _\DD \left( |\phi '(z)| \tau_\omega(z)\right)^p d\lambda_\omega (z)<\infty\}\quad (p>0).
$$
Note that ${\cal{B}}^{\omega,2} = {\cal {B}_2}$ is the classical Dirichlet space and it doesn't depend on the weight $\omega$. Note that since $\tau _\omega (z) = O((1-|z|^2))$, ${\cal{B}}^{\omega,p} = \{0\}$ whenever $p \leq 1$. Using standard arguments, one can easily prove that an analytic function $f$ on $\DD$ belongs to ${\cal{B}}^{\omega,p}$ if and only if 
$$
\displaystyle \sum _{n}\left ( \frac{\mu (R_n)}{A(R_n)}\right )^p< \infty, \  \mbox{where}\  d\mu (z)= \tau _\omega (z)|f'(z)|dA(z)\ \text{and}\ (R_n)_n\in \cL _\omega.
$$
In particular, this implies that the family $({\cal{B}}^{\omega,p} )_p$ is increasing.\\
The following result, which extends the main results in \cite {AFP, GP}, is a direct consequence of Theorem \ref {main result 1}.
\begin{cor}
Let $p\geq 1$ and let $\omega \in \cal{W}^\ast$. Let $\phi \in {\cal{B}}^\omega_0$. Then 
$$H_{\overline{\phi}}  \in {\cal{S}}_p(A^2_\omega)\quad \iff \quad \phi \in {\cal{B}}^{\omega,p}.$$
\end{cor}
As a consequence of Theorem \ref{main result 1}, we give some estimates of the singular values of compact Hankel operators. To this end, let us recall that $\mathcal{R}_{\phi ,\omega}(t) := \lambda _\omega( \{ z \in \mathbb{D}: \ \tau_\omega (z)|\phi ' (z)|  > t   \} )
$. 
Let $\mathcal {R}^+_{\phi,\omega}$ be the increasing rearrangement of the function $\tau_\omega |\phi '|$.
For any increasing function $h$, by a standard computation, we have 
$$
\displaystyle \int _\DD h\left ( \tau_\omega(z)|\phi '(z)|\right ) d\lambda_\omega (z)= \int_{0}^{\infty} \cR_{\phi,\omega}(t) dh(t).
$$
Then, there exists $B>0$ which depends only on $\omega$ such that 
\begin{equation}\label {10}
\displaystyle\sum_{n\geq 0}h \left (\frac{1}{B}\cR _{\phi,\omega}^+(n)\right) \leq \displaystyle \int _\DD h\left ( \tau_\omega(z)|\phi '(z)|\right ) d\lambda_\omega (z)\leq \displaystyle \sum_{n\geq 0}h(B\cR _{\phi,\omega}^+(n)).
\end{equation}
As a consequence of Theorem \ref{main result 1}, we obtain the following result.
\begin{thm}\label{equiv}
Let $\omega \in \cal {W}^*$ and let $\phi \in {\cal{B}}_0^\omega$. Let $\rho $ be an increasing function  such $\rho (x)/x^ \gamma$ is decreasing for some $\gamma \in (0,1)$. Then  
$$
 s_n(H_{\overline{\phi}} ) \asymp 1/ \rho (n) \Longleftrightarrow  \mathcal {R}^+_{\phi,\omega}(n)  \asymp 1/ \rho (n).
$$
\end{thm}
\begin{proof}
By Theorem \ref{main result 1} and inequalities (\ref{10}), there exists $B>0$ such that for every increasing  convex function $h$ we have
$$
\displaystyle\sum_{n\geq 0}h \left (\frac{1}{B}\cR _{\phi,\omega}^+(n)\right) \leq \displaystyle \sum _{n\geq 0} h(s_n(H_{\overline {\phi }}))\leq \displaystyle \sum_{n\geq 0}h(B\cR _{\phi,\omega}^+(n)).
$$ 
By Lemma \ref{Convex1}, we obtain the desired result.  
\end{proof}
\begin{thm}\label{majoration}
Let $\omega \in \cal {W}^*$ and let $\phi \in {\cal{B}}_0^\omega$. Let $\rho $ be an increasing function such that $\rho (x)/x^A$ is decreasing for some $A>0$. Suppose that $\cR _{\phi,\omega} ^+(x) = O(1/\rho (x))$, as $x\to \infty$. Then 
$$s_n(H_{\overline \phi}) = O(1/\rho (n)), \quad n\to \infty.$$
\end{thm}	
\begin{proof}
First, recall that $d\mu _\phi (z)= \tau _\omega ^2(z)|\phi '(z)|^2dA(z)$ and fix $( R_n)_n \in \cL _\omega$. By Inequality (\ref {HT}) it suffices to prove that $\lambda _n(T_{\mu _\phi}) = O(1/\rho ^2(n))$. By \cite{EE}, for any $p>0$ there exists $B>0$ such that, for any increasing function $h$ such that $h(t^p)$ is convex and $h(0)=0$, we have 
$$
\displaystyle \sum h(\lambda _n (T_{\mu _\phi})) \leq \displaystyle \sum h(B {\mu _\phi}(R_n)/A(R_n)).
$$
Let $b>1$ such that  $((bR_n)_n)\in \cL _\omega$.  By subharmonicity of $|\phi '|^{2/p} $, we have 
$$
\tau ^2_\omega (z) |\phi '(z)|^2\leq \left ( \frac{C(b,p,\omega) }{A(R_n)}\displaystyle \int _{bR_n}\tau ^{2/p}_\omega (\zeta) |\phi '(\zeta)|^{2/p}dA(\zeta)\right )^p,\quad \text {for all}\ z\in R_n.
$$
We obtain from the convexity of $h_p(t)=: h(t^p)$ that 
\[
\begin{array}{lll}
 \displaystyle \sum h(B {\mu _\phi}(R_n)/A(R_n)) &\leq & \displaystyle \sum h _p\left ( \frac{C(b,p,\omega) }{A(R_n)}\displaystyle \int _{bR_n}\tau ^{2/p}_\omega (\zeta) |\phi '(\zeta)|^{2/p}dA(\zeta)\right )\\
 \\
 & \leq & \displaystyle \sum  \frac{1}{A(R_n)}\displaystyle \int _{bR_n}h(C(b,p,\omega)  \tau ^{2}_\omega (\zeta) |\phi '(\zeta)|^{2})dA(\zeta)\\
 \\
 & \lesssim & \displaystyle \int _{\DD}h( C(b,p,\omega)  \tau ^{2}_\omega (\zeta) |\phi '(\zeta)|^{2})d\lambda _\omega (\zeta).\\
 \end{array}
 \]
Combining these inequalities with (\ref {10}) and the hypothesis that $\cR^+_{\phi,\omega}(n)\lesssim 1/\rho(n)$ we obtain 
$$
\displaystyle \sum h(\lambda _n( (T_{\mu _\phi})) \leq \displaystyle \sum_{n\geq 0}h(B\left ( \cR _{\phi,\omega}^+(n)\right )^2) \leq  \displaystyle \sum_{n\geq 0}h(B/\rho ^2(n)),
$$
where $B$ depends on $\omega, b, p$ and $A$. The result comes from Lemma \ref{Convex3}.
\end{proof}

\section{The cut-off }

In this section we consider weights $\omega \in {\cal {W}^\ast}$ such that  $\tau _\omega$ is equivalent to a radial function. We cite as examples, radial weights $\omega \in \cal{W}^\ast$, positive harmonic weights and  weights $\omega = e^{-\varphi} \in \cal{W}_0$ such that $\Delta \varphi $ is equivalent to a radial function.
\begin{proof}[Proof of Theorem \ref {main result 3}]
Suppose that $s_n(H_{\overline{\phi}}) = o (s_n(H_{\overline{z}}))$. Let $\delta \in (0,1)$ and let $h_\delta (t)= (t-\delta)^+$. By Theorem \ref{main result 1} we have
$$
\displaystyle \int _{\mathbb{D}}h_\delta \left (\frac{1}{B}\tau _\omega (z)|\phi '(z)|\right )d\lambda _\omega(z)\leq \displaystyle \sum _{n\geq 1}h_\delta (s_n(H_{\overline{\phi}})).
$$
Let $\rho \in (1/2,1)$ and put $K= \displaystyle \int _0^{2\pi}|\phi '(\rho e^{it})|\frac{dt}{2\pi}$ . By Jensen's inequality we have
\[
\begin{array}{lll}
\displaystyle \int _{\mathbb{D}}h_\delta \left (\frac{1}{B}\tau _\omega (z)|\phi '(z)|\right )d\lambda _\omega(z)&\geq &
2\displaystyle \int _0 ^1h_\delta \left (\frac{1}{B}\tau _\omega (r)\displaystyle \int _0^{2\pi}|\phi '(r e^{it})|\frac{dt}{2\pi}\right )\frac{rdr}{\tau _\omega ^2(r)}\\
&\geq &
\displaystyle \int _\rho ^1h_\delta \left (\frac{1}{B}\tau _\omega (r)\displaystyle \int _0^{2\pi}|\phi '(\rho e^{it})|\frac{dt}{2\pi}\right )\frac{dr}{\tau _\omega ^2(r)}\\
& = & \displaystyle \int _\rho ^1\left (\frac{K}{B}\tau _\omega (r) - \delta \right )^+\frac{dr}{\tau _\omega ^2(r)}.
\end{array}
\]
Suppose that $\phi ' \neq 0$, then $K >0$. For $\tau _\omega (r) \geq \frac{2B}{K}\delta $, we have $\left (\frac{K}{B}\tau _\omega (r) - \delta \right )^+\geq \frac{K}{2B}\tau _\omega (r)$. Then we obtain 
\begin{equation}\label{E41}
\displaystyle \int _{\mathbb{D}}h_\delta \left (\frac{1}{B}\tau _\omega (z)|\phi '(z)|\right )d\lambda _\omega(z)\geq  \frac{K}{2B}\displaystyle \int _{\{r \in (\rho, 1):\ \tau _\omega (r) \geq \frac{2B}{K}\delta \}}\frac{dr}{\tau _\omega(r)}.
\end{equation}
 Let $\varepsilon \in (0, K/4B^2)$ and let $N$ be such that for $n\geq N$ we have $s_n(H_{\overline{\phi}}) \leq \varepsilon s_n(H_{\overline{z}})$. Using Theorem \ref {main result 1}, we have
\[
\begin{array}{lll}
\displaystyle \sum _{n\geq 1}h_\delta (s_n(H_{\overline{\phi}})) & \leq  & \displaystyle \sum _{n< N}h_\delta (s_n(H_{\overline{\phi}}))+\displaystyle \sum _{n\geq 1}h_\delta (\varepsilon s_n(H_{\overline{z}}))\\
& \leq & N \| H_{\overline{\phi}}\|+ \displaystyle \int _{\mathbb{D}}h_\delta (\varepsilon B \tau _\omega (z))d\lambda _\omega(z)\\
& \leq &  N \| H_{\overline{\phi}}\| +2\varepsilon B \displaystyle \int _{ \{r\in (0,1): \ \tau_\omega (r) \geq \frac{\delta}{\varepsilon B  } \} }\frac{dr}{\tau_\omega (r)}.\\
\end{array}
\]
Since $1/\varepsilon B > 2B/K$, we obtain 
\begin{equation}\label{E42}
\displaystyle \sum _{n\geq 1}h_\delta (s_n(H_{\overline{\phi}}) )\leq N \| H_{\overline{\phi}}\| +2\varepsilon B \displaystyle \int _{\{r\in (0,1): \ \tau_\omega (r) \geq \frac{2B}{K}\delta\} } \frac{dr}{\tau_\omega(r)}.
\end{equation}
Combining inequalities \eqref{E41} and \eqref{E42}, we obtain
$$
 \frac{K}{2B}\displaystyle \int _{\{r \in (\rho, 1):\ \tau _\omega (r) \geq \frac{2B}{K}\delta \}}\frac{dr}{\tau _\omega(r)} \leq N \| H_{\overline{\phi}}\| +2\varepsilon B \displaystyle \int _{\{r: \ \tau_\omega (r) \geq \frac{2B}{K}\delta\} } \frac{dr}{\tau_\omega (r)}.
$$
Since $\displaystyle \int ^1\frac{dr}{\tau _\omega (r)}=\infty$, when $\delta$ goes to $0$, we obtain  $\frac{K}{2B} \leq 2\varepsilon B$. This gives a contradiction.\\
The second assertion is obtained by using the same argument.
\end{proof}


\begin{cor}
Let $\omega \in \cal {W}^*$ be such that  $\tau ^2_\omega (z)\asymp  (1-|z|^2)^{2+\beta}\nu ^2(\log  (\frac{1}{1-|z|^2}))$, where $\beta \geq 0$ and $\nu $ is a monotone function such that $\nu (2t) \asymp \nu (t)$. Let $\phi \in {\cal{B}}_0^\omega$ and let $p = \frac{2(1+\beta)}{2+\beta}$. If 
$$s_n(H_{\overline{\phi}}) = o \left ( \frac{1}{n^{\frac{1}{p}}\nu ^{\frac{1}{1+\beta}}(\log n )}\right ),$$ then $\phi ' =0.$
\end{cor}
\begin{proof}
It is not difficult to verify that 
$$
\displaystyle \int _\delta ^1\frac{dt}{t^{2+\beta} \nu ^2(\log 1/t)} \asymp \frac{1}{\delta^{1+\beta }\nu ^2(\log 1/\delta)}.
$$
Then, since $\tau ^2_\omega (z)\asymp  (1-|z|^2)^{2+\beta}\nu ^2(\log  (\frac{1}{1-|z|^2})$, we have 
\[
\begin{array}{lll}
\cR_{z,\omega} (t) &=& \displaystyle \int _{\{z\in \DD: \ \tau _\omega (z) \geq t\} }\frac{dA(z)}{(1-|z|^2)^{2+\beta} \nu ^2(\log 1/(1-|z|^2))} \\ \\
& \asymp & \Large \frac{1}{t^{\frac{2(1+\beta)}{2+\beta}} \nu ^{\frac{2}{2+\beta}}(\log 1/t)}.\\
\end{array}
\]
This implies that $\cR_{z, \omega} ^{+}(n)\asymp \frac{1}{n^{\frac{1}{p}}\nu ^{\frac{1}{1+\beta}}(\log n )}$. By the second assertion of Theorem \ref{main result 3}, we get $\phi ' = 0$. 
\end{proof}
Note that if $\omega $ is  radial then 
$$
H_{\overline{z}}^*H_{\overline{z}} \left (\frac{z^n}{\|z^n\|} \right )=\left ( \frac{\|z^{n+1}\|^2}{\|z^{n}\|^2}-\frac{\|z^{n}\|^2}{\|z^{n-1}\|^2} \right )\frac{z^n}{\|z^n\|}=: m^2_\omega (n)\frac{z^n}{\|z^n\|} ,\quad n\geq 1.
$$ 
So, the sequence of the singular values of $H_{\overline{z}}$ is the decreasing rearrangement of  the sequence $(m_\omega (n))_{n\geq 1}$.\\
For the standard Bergman spaces $A^2_\alpha$, it is easy to see that 
$$
\|z^n\|^2=\frac{ \Gamma (n+1)\Gamma(\alpha+2)}{\Gamma(n+\alpha+2)}\ \mbox{and}\ 
m^2_{\omega _\alpha}(n) = \frac{\alpha+1}{(n+\alpha+1)(n+\alpha+2)}.
$$
where $\Gamma (x)= \displaystyle \int _0^\infty t^{x-1}e^{-t}dt$ denotes the Gamma function.\\
Then 
$$
s_n(H_{\overline{z}})\sim \frac{\sqrt{\alpha+1}}{n+1}.
$$
For larger Bergman spaces we have the following result. 
\begin{prop}\label{beta}
Let  $\omega \in \cal {W}^*$ and let $\beta >0$. Suppose that  $\tau _\omega^2 (z) \asymp (1-|z|^2)^{2+\beta}\nu ^2(\log ( \frac{1}{1-|z|^2}))$, where $\nu$ is a monotone function which satisfies $\nu (2t) \asymp \nu (t)$. Then 
$$
s_n(H_{\bar {z}}) \asymp  \frac{1}{n^{\frac{1}{p}}\nu ^{\frac{1}{1+\beta}}(\log n )},\quad \text{where}\  p = \frac{2(1+\beta)}{2+\beta}.
$$
\end{prop}
\begin{proof}
From the hypothesis, we have $\cR_{z, \omega} ^{+}(n)\asymp \frac{1}{n^{\frac{1}{p}}\nu ^{\frac{1}{1+\beta}}(\log n )}$.
By Theorem \ref{equiv}, we obtain the result.
\end{proof}
\section{Critical decay }\label{CD}
In this section we describe the class of functions $\phi \in {\cal{B}}_0^\omega$ such that $s_n(H_{\overline{\phi}}) = O(s_n(H_{\overline{z}})) $. This kind of problem was first examined by Arazy, Fisher and Peetre for standard Bergman spaces  \cite{AFP}. They proved that if $\phi$ belongs  to the Besov space $\mathcal{B}_1$ given by $$\mathcal{B}_1:= \lbrace f\in \text{Hol}(\mathbb{D}),\quad \int_\mathbb{D} |f''(z)| dA(z)<\infty\rbrace,$$ then 
\begin{equation*}\label{mean}
\sigma_n(H_{\overline{\phi}})= O\left( \log(n+1)\right), \ \text{where}\  \sigma_n(H_{\overline{\phi}}):= \displaystyle \sum _{j=1}^ns_j(H_{\overline{\phi}}).
\end{equation*}
 They also proved, in the same paper, that the converse is false. In \cite {ER}, M. Engli\v{s} and R. Rochberg gave a complete answer to this problem for the classical Bergman space. They proved, by using Boutet de Monvel-Guillemin theory, that $s_n(H_{\overline{\phi}})= O\left( 1/n \right) $ if and only if $\phi '\in H^1$. This result was extended by R. Tytgat \cite{Tyt,Tyt1} to standard Bergman spaces $A^2_\alpha$. In what follows we study this problem in more general situations. Our approach is 
based on Theorem \ref{main result 1}. 
\begin{prop}\label{NCHp}
Let $\omega \in \cal {W}^*$ be a weight such that $\tau ^2_\omega (z)\asymp  (1-|z|^2)^{2+\beta}\nu ^2(\log  (\frac{1}{1-|z|^2})$ where $\beta \geq 0$ and $\nu $ is a monotone function such that $\nu (2t) \asymp \nu (t)$. Let $\phi \in {\cal{B}}_0^\omega$ and let $p= \frac{2(1+\beta)}{2+\beta}$. Then $$s_n(H_{\overline{\phi}}) = O\left ( \frac{1}{n^{\frac{1}{p}}\nu ^\frac{1}{1+\beta}(\log n)} \right )\quad \Longrightarrow \quad \phi ' \in H^{p}.$$
\end{prop}
\begin{proof}
Suppose that $s_n(H_{\overline{\phi}})  \leq  \frac{C}{n^{\frac{1}{p}}\nu ^\frac{1}{1+\beta}(\log n)} $. Since  $\tau ^2_\omega (z)\asymp  (1-|z|^2)^{2+\beta}\nu ^2(\log  (\frac{1}{1-|z|^2})$, we have 
$$
\cR_{z,\omega} (t) \asymp \frac{1}{t^p\nu ^{\frac{2}{2+\beta}}(\log (1/t))}\quad \text{and}\quad \cR_{z,\omega}^+(n) \asymp \frac{1}{n^{\frac{1}{p}}\nu ^{\frac{1}{1+\beta}}(\log (n))}.
$$
Then for every increasing function $h$, we have
$$
\displaystyle \sum _{n=1}^\infty h\left (\frac{1}{B_1n^{\frac{1}{p}}\nu ^{\frac{1}{1+\beta}}(\log (n))} \right )\leq \displaystyle \int _{\mathbb D}h(\tau _\omega(z) ) d\lambda _\omega (z) \leq \displaystyle \sum _{n=1}^\infty h\left (\frac{B_1}{n^{\frac{1}{p}}\nu ^{\frac{1}{1+\beta}}(\log (n))}\right ).
$$
where $B_1$ doesn't depend on $h$.  If in addition $h$ is convex and $h(0)=0$, then by Theorem\ref{main result 1} we obtain
\[
\begin{array}{lll}
\displaystyle \int _{\mathbb D}  h\left (\tau _\omega (z)|\phi ' (z)|\right )d\lambda _\omega(z) 
& \lesssim & \displaystyle \sum _{n=1}^\infty h \left ( B s_n(H_{\overline \phi})\right )\\
& \lesssim & \displaystyle \sum _{n=1}^\infty h \left ( BC \cR_{z ,\omega} ^{+}(n) \right )\\
 & \lesssim & \displaystyle \int _{\mathbb D}  h\left (B_1BC\tau _\omega (z) \right )d\lambda _\omega(z). \\

\end{array}
\]
Let $\varepsilon \in (0,1)$ and put $p_\varepsilon = (1-\varepsilon)p +\varepsilon$. Note that if $\beta = 0$ (i.e. $p =1$) then $p_\varepsilon =1$ and if $\beta >0$ then $1 < p_\varepsilon < p$. Note that in the two cases we have 
$$
\displaystyle \int ^1\frac{dr}{\tau _\omega ^{2-p_\varepsilon}(r) }= \infty.
$$
The last inequality, with $h(t) = h_\delta (t^{p_\varepsilon})$, becomes
$$
\displaystyle \int _{\mathbb D}  h_\delta\left (\tau _\omega ^{p_\varepsilon }(z)|\phi ' (z)|^{p_\varepsilon}\right )d\lambda _\omega(z)
\lesssim  \displaystyle \int _{\mathbb D}  h_\delta \left (K^{p_\varepsilon }\tau _\omega^{p_\varepsilon }(z) \right )d\lambda _\omega(z), \quad \text{where}\  K = CBB_1. $$
Using the convexity of $h_\delta$, we get   
$$
\displaystyle \int _0^1 h_\delta \left (\tau _\omega ^{p_\varepsilon }(r)\displaystyle \int _0^{2\pi}|\phi ' (re^{it})|^{p_\varepsilon}\frac{dt}{2\pi}\right )\frac{dr}{\tau _\omega ^{2}(r) } \leq \displaystyle \int _{\mathbb D}  h_\delta\left (\tau _\omega ^{p_\varepsilon }(z)|\phi ' (z)|^{p_\varepsilon}\right )d\lambda _\omega(z)  \lesssim  \displaystyle \int _{\mathbb D}  h_\delta\left (K^{p_\varepsilon }\tau _\omega^{p_\varepsilon }(z) \right )d\lambda _\omega(z).
$$
Suppose that there exists $\rho \in (0,1)$ such that $\| \phi _\rho '\| _{p_\varepsilon }\geq 2K$. We have 
$$
\displaystyle \int _\rho^1 h_\delta \left (\tau _\omega^{p_\varepsilon }(r)\| \phi _\rho '\| _{p_\varepsilon }^{p_\varepsilon }\right )\frac{dr}{\tau _\omega ^{2}(r) } \lesssim \displaystyle \int _{\mathbb D}  h_\delta\left (K^{p_\varepsilon }\tau _\omega ^{p_\varepsilon }(z) \right )d\lambda _\omega (z).$$
Now, using the fact that $h_\delta (t) \geq t/2$ if $t\geq 2\delta$, we get
$$
\| \phi _\rho '\| _{p_\varepsilon }^{p_\varepsilon }\displaystyle \int _{\{r\in (\rho ,1):\ \tau _\omega ^{p_\varepsilon }(r)\| \phi _\rho '\| _{p_\varepsilon }^{p_\varepsilon } \geq 2\delta\} }\frac{dr}{\tau _\omega ^{2-p_\varepsilon}(r) } \lesssim K^{p_\varepsilon }\displaystyle \int _{\{r\in (0 ,1):\ \tau _\omega^{p_\varepsilon }(r)K^{p_\varepsilon } \geq \delta\} }\frac{dr}{\tau _\omega^{2-p_\varepsilon}(r) }.
$$
Since $\| \phi _\rho '\| _{p_\varepsilon } \geq 2K$, $ \{r\in (\rho ,1):\ \tau _\omega ^{p_\varepsilon }(r)K^{p_\varepsilon } \geq \delta\} \subset \{r\in (\rho ,1):\ \tau _\omega ^{p_\varepsilon }(r)\| \phi _\rho '\| _{p_\varepsilon }^{p_\varepsilon } \geq 2\delta\}$. Then 
$$
\| \phi _\rho '\| _{p_\varepsilon }^{p_\varepsilon }\displaystyle \int _{\{r\in (\rho ,1):\ \tau _\omega ^{p_\varepsilon }(r)K^{p_\varepsilon } \geq \delta\} }\frac{dr}{\tau _\omega ^{2-p_\varepsilon}(r) } \lesssim K^{p_\varepsilon }\displaystyle \int _{\{r\in (0 ,1):\ \tau _\omega ^{p_\varepsilon }(r)K^{p_\varepsilon } \geq \delta\} }\frac{dr}{\tau _\omega  ^{2-p_\varepsilon}(r) }.
$$
Recall that $ \displaystyle \int ^1\frac{dr}{\tau _\omega ^{2-p_\varepsilon}(r) }= \infty$. We obtain, when $\delta $ goes to $ 0^+$, that $\| \phi _\rho '\| _{p_\varepsilon } \lesssim K$. When $\varepsilon $ goes to $0$, we get $\| \phi _\rho '\| _{p} \lesssim K$. This proves that $\phi ' \in H^p$.
\end{proof}

\begin{proof}[Proof of the first assertion of Theorem \ref {main result 2}.]
Let $p=p_\beta (:= \frac{2(1+\beta)}{2+\beta})$. Suppose that $s_n(H_{\overline {\phi}}) = O(1/n^{1/p} )$. By Proposition \ref{NCHp}, with $\nu =1$, we have $\phi '\in H^p$.\\ 
For the converse, by Theorem \ref{majoration}, it suffices to prove that $\cR _{\phi, \omega}^+ (x)= O(1/x^{1/p})$, whenever $\phi ' \in H^{p}$. To this end, let $U$ be the non-tangential maximal function of $\phi'$. Since $\phi ' \in H^p$ we have $U\in L^p$ and $\| U\| _{p}\lesssim \|\phi '\| _p$ (even if $p=1$).
We have 
\[ 
\begin{array}{lll}
\cR _{\phi, \omega}(t) &= & \lambda _\omega  \left ( \{z\in \DD : \tau_\omega (z) |\phi '(z)|\geq t\} \right )\\
\\
& \leq &  \lambda _\omega \left ( \left \{re^{i\theta} \in \DD : \tau_\omega (r) U(e^{i\theta})\geq t \right \} \right )\\
\\
& \asymp & \displaystyle \int _{ \{re^{i\theta} \in \DD : \tau_\omega (r) U(e^{i\theta})\geq t \} }\frac{drd\theta}{\pi(1-r^2)^{2+\beta}}\\
\\
&  \lesssim & \frac{1}{t^p} \displaystyle \int _0^{2\pi}U^p(e^{i\theta})\frac{d\theta}{2\pi}\\
\\
& \lesssim & \frac{1}{t^p}\| \phi '\| _p^p.
\end{array}
\]
This is equivalent to $\cR _{\phi, \omega}^+ (x)\lesssim \frac{\| \phi '\|_p}{x^{1/p}}$. The proof is complete.
\end{proof}
Now, we study the converse of Proposition \ref {NCHp}, when $\tau ^2_\omega (z)\asymp  (1-|z|^2)^{2+\beta}\nu ^2(\log  (\frac{1}{1-|z|^2}))$. In the following result we consider the case $\beta >0$. The case $\beta = 0$, will be discussed right after.
\begin{thm}\label{Hp}
Let $\omega \in \cal {W}^*$ be a weight such that $\tau ^2_\omega (z)\asymp  (1-|z|^2)^{2+\beta}\nu ^2(\log  (\frac{1}{1-|z|^2}))$, where $\beta > 0$ and $\nu $ is a monotone function such that $\nu (2t) \asymp \nu (t)$. Let $\phi \in {\cal{B}}_0^\omega$ and let $p= \frac{2(1+\beta)}{2+\beta}$. Then 
$$s_n(H_{\overline{\phi}}) = O\left ( \frac{1}{n^{\frac{1}{p}}\nu ^\frac{1}{1+\beta}(\log n)} \right )\quad \iff \quad \phi ' \in H^{p}.$$
In this case we have 
$$
s_n(H_{\overline{\phi}}) \lesssim \frac{\| \phi '\| _p}{n^{\frac{1}{p}}\nu ^\frac{1}{1+\beta}(\log n)},
$$
where the involved constant doesn't depend on $\phi$.
\end{thm}
\begin{proof}
By Proposition \ref{NCHp}, It remains to prove that if  $\phi ' \in H^p$ then  $s_n(H_{\overline{\phi}}) = O\left ( \frac{1}{n^{\frac{1}{p}}\nu ^\frac{1}{1+\beta}(\log n)} \right )$. We will proceed as in the proof of Theorem \ref {main result 2}. Without loss of generality, suppose that $\| \phi '\| _p =1$. Our goal is to show that $$\cR _{\phi, \omega }(t) = O\left (\frac{1}{t^p \nu ^{\frac{2}{2+\beta}}(\log (1/t))} \right ),$$
where the implied constant doesn't depend on $\phi $.\\
We have
\begin{equation}\label{rt}
|\phi' (r\zeta)| \leq \frac{1}{(1-r)^{\frac{1}{p}}}=\frac{1}{(1-r)^{\frac{2+\beta}{2(1+\beta)}}},\quad  r\in (0,1)\ \text{and}\ \zeta \in \mathbb {T}.
\end{equation}
This implies that 
$$
|\phi' (r\zeta)|\tau _\omega (r) \lesssim  (1-r)^{\frac{(2+\beta)\beta}{2(1+\beta)}}\nu (\log (1/1-r)).\\
$$
Then there exists $r_0 \in (0,1)$, which depends only on $\omega$, such that $$
|\phi' (r\zeta)|\tau _\omega (r) \leq  (1-r)^{\frac{(2+\beta)\beta}{4(1+\beta)}},\quad r \in (r_0, 1).
$$
So, if $|\phi' (r\zeta)|\tau _\omega (r) \geq t$ then  $r \leq r_t$ where $r_t $ is given by $(1-r_t)^{\frac{(2+\beta)\beta}{4(1+\beta)}}= t$.\\
Let $U_t$ be the non-tangential maximal function associated with $z\in \mathbb D \to  \phi' (r_tz)$. By inequality (\ref {rt}), we have 
\begin{equation}\label{rt2}
|\phi' (r\zeta)|\leq U_t(\zeta)\leq \frac{1}{(1-r_t)^{1/p}}  = \frac{1}{t^{2/\beta}},\quad r \leq r_t.
\end{equation}
For $t$ small enough, we have 
\begin{align*}
	\cR_{\phi,\omega}(t) &=  \lambda _\omega( \{ re^{i\theta} \in \mathbb{D} : \tau _\omega (r)|\phi ' (re^{i\theta})| \geq t \})\\
	&\lesssim \int_{ \{  re^{i\theta}: r\in (0,r_t),\ \tau _\omega (r)U_t(e^{i\theta}) \geq t \} }  \frac{dr}{\tau _\omega ^2(r)} d\theta\\
	&\lesssim  \displaystyle \int _{0}^{2\pi} \int_{ \{  r\in (0,r_t): \tau_\omega (r) \geq \frac{t}{U_t(e^{i\theta})} \} }  \frac{dr}{\tau _\omega ^2(r)} d\theta\\
	& \lesssim  \displaystyle \int _{0}^{2\pi}\cR_{z,\omega}\left (\frac{ t}{U_t (e^{i\theta})}\right )d\theta.
	\end{align*}
Since $\tau _\omega ^2 \asymp (1-r^2)^{2+\beta}\nu^2\left (\log (1/1-r^2)\right)$, we have $
\cR_{z,\omega}(t)\asymp \frac{1}{t^p\nu ^{\frac{2}{2+\beta}}(\log 1/t)}.
$
First, note that if $U_t(e^{i\theta}) \leq t^{1/2}$ then $\cR_{z,\omega}\left (\frac{ t}{U_t (e^{i\theta})}\right )\leq \cR_{z,\omega} (t^{1/2})\leq \cR_{z,\omega} (t)$.
Otherwise, from equation (\ref {rt2}), we get
$$
\cR_{z,\omega}\left (\frac{ t}{U_t (e^{i\theta})}\right ) \asymp  \frac{U_t^p(e^{i\theta})}{t^p\nu ^{\frac{2}{2+\beta}}(\log U_t(e^{i\theta})/t)}\lesssim   \frac{U_t^p(e^{i\theta})}{t^p\nu ^{\frac{2}{2+\beta}}(\log 1/t)}.
$$
Combining these inequalities we obtain
$$
\cR_{\phi,\omega}(t) \lesssim \frac{\| U_t \| _p^p}{t^p\nu ^{\frac{2}{2+\beta}}(\log 1/t)}\lesssim \frac{1}{t^p\nu ^{\frac{2}{2+\beta}}(\log 1/t)},
$$
where the involved constant doesn't depend on $ \phi' $. By Theorem \ref {majoration}, we obtain the desired result.
\end{proof}
Now, we will prove that the previous result is not true when $\beta = 0$ which is somewhat unexpected. 
\begin{prop}\label{CE}
Let $\omega \in \cW ^\ast$ such that $\tau _\omega (z) \asymp \frac{(1-|z|)}{\log ^\alpha (1/1-|z|)}$ with $\alpha >0$. Then for $\nu \in ]\frac{\alpha +1}{2\alpha +1},1[$, there exists $\phi \in {\mathcal B}^\omega _0$ such that $\phi ' \in H^1$ and 
$$s_n(H_{\overline{\phi}}) \asymp \frac{1}{n ^\nu}.$$
\end{prop}
\begin{proof}
Let $\gamma \in (1, \alpha +1)$ be such that $\nu = \frac{\alpha +\gamma}{2\alpha+1}$. Let $\phi$ be such that 
$$\phi ' (z) = \frac{1}{(1-z)\log ^\gamma (\frac{e}{1-z})},\quad  z\in \DD.$$
Since $\gamma >1,\  \phi ' \in H^1$. Write $\cR_{\phi,\omega}=\cR _1+\cR_2$, where 
$$
\cR_1(t) =  \lambda _\omega( \{ re^{i\theta} \in \mathbb{D} : \ 1-r>|\theta|\ \text{and }\ \tau _\omega (r)|\phi ' (re^{i\theta})| \geq t \})
$$
and 
$$
\cR_2(t) =  \lambda _\omega( \{ re^{i\theta} \in \mathbb{D} : \ 1-r\leq |\theta|\ \text{and }\ \tau _\omega (r)|\phi ' (re^{i\theta})| \geq t \}).
$$
Clearly, we have 
$$
\cR_1(t) \asymp  \lambda _\omega( \{ re^{i\theta} \in \mathbb{D} : \ 1-r>|\theta|\ \text{and }\  \frac{1}{\log ^{\alpha+\gamma}(e/(1-r))} \geq t \})\asymp \frac{1}{t ^{\frac{2\alpha+1}{\alpha +\gamma}}}.
$$
Similarly, 
$$\cR_2(t) \asymp  \lambda _\omega( \{ re^{i\theta} \in \mathbb{D} : \ 1-r\leq |\theta|\ \text{and }\  \frac{1-r}{\log ^\alpha (1/1-r)} \frac{1}{|\theta|\log ^{\gamma}(e/|\theta|)} \geq t \})\asymp \frac{1}{t ^{\frac{2\alpha+1}{\alpha +\gamma}}}.$$
Then, $\cR_{\phi,\omega}(t) \asymp  \frac{1}{t ^{\frac{2\alpha+1}{\alpha +\gamma}}}$. 
By Theorem \ref {equiv}, we obtain that $s_n(H_{\overline{\phi}}) \asymp \frac{1}{n ^{\nu}}$. The proof is complete.
\end{proof}
\section{Asymptotics}
Now, we will precise the results obtained in Section \ref{CD} when $s_n(H_{\overline{z}})$ is regular. The main result of this section is the following theorem.
\begin{thm}\label{asymptotics}
Let $\omega \in \cal {W}^*$ be a radial weight such that $\tau ^2_\omega (z)\asymp  (1-|z|^2)^{2+\beta}\nu ^2(\log  (\frac{1}{1-|z|^2}))$, where $\beta > 0$ and $\nu $ is a monotone function such that $\nu (2t) \asymp \nu (t)$.  Suppose that 
$$
s_n(H_{\overline{z}}) \sim \frac{\gamma}{n^{\frac{1}{p}}\nu ^\frac{1}{1+\beta}(\log n)},
$$
where $\gamma > 0$ and $p= \frac{2(1+\beta)}{2+\beta}$. Then, for $\phi ' \in H^{p}$, we have
$$s_n(H_{\overline{\phi}}) \sim  \frac{\gamma}{n^{\frac{1}{p}}\nu ^\frac{1}{1+\beta}(\log n)}\| \phi '\|_p.$$
\end{thm}
To prove this result, we will introduce the following functionals (see \cite{BS, Pus}).
Let $ T$ be a compact operator between two Hilbert spaces. Let $n(s,T)$ be the singular values counting function given by
$$
n(s,T)=\# \{ n: \ s_n(T) \geq s\} ,\quad s>0.
$$
The class of  strictly increasing continuous functions $\psi:(0,+\infty)\to (0,+\infty)$  such that $\psi(0)=0$ and such that
$$
\psi(\alpha t)\sim \alpha^p \psi(t),\quad(t\to 0^+),
$$
for some $p> 0$, will be denoted by $\cC _p$. Let
$$
D _\psi (T):=\displaystyle\limsup_{s\to 0^+}\psi(s)n(s,T)\quad\mbox{and}\quad d _\psi (T):=\displaystyle\liminf_{s\to 0^+}\psi(s)n(s,T).
$$
From these definitions it is easy to see that if $\psi \in \cC _p$, then 

$$d_\psi (T) = D_\psi (T) \in (0,\infty) \Longrightarrow s_n(T) \sim D_\psi (T)^{1/p}\psi ^{-1}(1/n).$$
For more definitions and  properties of these functionals see Section \ref{Annexe}.\\

First, we need some results on the operator $H_{\overline{z}}P_\omega M_g$, where $M_g$ denotes the multiplication operator defined on $L^2 (dA_\omega )$.\\

\noindent For an arc $\delta=\{e^{i\theta}:\quad \theta_1\leq \theta<\theta_2\}\subset \TT$ , $R_N^\delta$ will denote
$$
R_N^\delta=\{z=re^{i\theta}:\quad 0<1-r \le 2\pi/N,\;    \theta_1\leq \theta<\theta_2\},\quad N=1,2,...
$$
The proof  of the following results is similar to that one given by A. Pushnitski in \cite {Pus}.\\
\begin{lem}\label{lem1} Let $\omega \in \cal {W}^*$ be a radial weight. Let $\psi \in {\cal C}_p$ such that $D_{\psi}(H_{\overline{z}})$ is finite. The following are true
\begin{enumerate}
\item Let $g$ be a bounded function on $\mathbb{D}$. We have
	$$
	D_{\psi}(H_{\overline{z}}P_\omega M_g)\leq \|g\|_{\infty}^p D_{\psi}(H_{\overline{z}}).
	$$
	\item Let $\delta_1, \delta _2$ be two arcs of $\TT$ such that $\overline{\delta_1}\cap\overline{ \delta _2} = \emptyset$. Then 
$$
	\left( H_{\overline{z}}P_\omega M_{\chi_{R_N^{\delta _1}}}\right)^\ast\left( H_{\overline{z}}P_\omega M_{\chi_{R_N^{\delta _2}}}\right)\in \displaystyle \cap _{p>0} \cal{S}_p.
$$

\item Let $\delta\subset\partial\mathbb{D}$ be an arc such that $|\delta|<2\pi$. Then
	$$
	D_{\psi}(H_{\overline{z}}P_\omega {\chi_{R_N^\delta}})\leq |\delta| D_{\psi}(H_{\overline{z}}).
	$$ 
	\item Let $\delta_1, \delta _2$ be two arcs of $\TT$ such that $\overline{\delta_1}\cap\overline{ \delta _2}$ is reduced to one point. Then 
$$
	\left( H_{\overline{z}}P_\omega M_{\chi_{R_N^{\delta _1}}}\right)^\ast\left( H_{\overline{z}}P_\omega M_{\chi_{R_N^{\delta _2}}}\right)\in \displaystyle \Sigma_{\psi \circ \sqrt {}}^0 .
$$

	\end{enumerate}
\end{lem}

\begin{proof}
	Since $g$ is bounded, $M_g$ is bounded on $L^2_\omega$ and $\| M_g\| = \|g\|_{\infty}$. Then $n(s, H_{\overline{z}}P_\omega M_g)\leq n(s, \|M_g\| H_{\overline{z}})= n(s, \|g\| _\infty H_{\overline{z}})$, for all $s>0$.
	Hence, by Proposition \ref{prop ABC}, we get
$$
		D_{\psi}(H_{\overline{z}}P_\omega M_g) \leq D_{\psi}(\|g\|_\infty H_{\overline{z}})
		 = \|g\|_{\infty}^p D_{\psi}(H_{\overline{z}}).
$$
This proves the first assertion. \\
The same proof as that proposed by Pushnitski in \cite{Pus}, gives the second assertion.\\
	To prove the third assertion, let $N\in\mathbb{N}^\ast$ such that $\frac{2\pi}{N+1}\le |\delta|<\frac{2\pi}{N}$ and let 
	$$
	\delta_n:=e^{2\pi in/N}\delta,\quad n=1,2...,N.
	$$ 
	Let $R_N^n:=R_N^{\delta_n}$. Since $g:=\displaystyle\sum_{n=1}^{N}\chi_{R_N^n}\leq 1$, we have
	\begin{align*}
		D_{\psi}(H_{\overline{z}})\geq D_{\psi}(H_{\overline{z}}P_\omega M_g)
		= D_{\psi}\left( \displaystyle\sum_{n=1}^{N}H_{\overline{z}}P_\omega M_{\chi_{R_N^n}}\right). 
			\end{align*}
	For $i\neq j$ we have $\left( H_{\overline{z}}P_\omega M_{\chi_{R_N^i}}\right)\left( H_{\overline{z}}P_\omega M_{\chi_{R_N^j}}\right)^\ast=0 $. Since the closures of $R_N^i$ and $R_N^j$ are disjoint, we have
	$$
	\left( H_{\overline{z}}P_\omega M_{\chi_{R_N^i}}\right)^\ast\left( H_{\overline{z}}P_\omega M_{\chi_{R_N^j}}\right)\in \Sigma^0_{\psi\circ\sqrt{.}}.
	$$
	Further, since $\omega$ is radial, the operators $H_{\overline{z}}P_\omega M_{\chi_{R_N^j}}$ are unitarily equivalent. It follows from Theorem \ref{orthogonal} that
	\begin{align*}
	D_{\psi}(H_{\overline{z}})
	\geq D_{\psi}\left( \displaystyle\sum_{n=1}^{N}H_{\overline{z}}P_\omega M_{\chi_{R_N^n}}\right) =N D_{\psi}(H_{\overline{z}}P_\omega M_{\chi_{R_N^\delta}}).
	\end{align*}
	Therefore, $D_{\psi}(H_{\overline{z}}P_\omega M_{\chi_{R_N^\delta}}) \leq 1/N D_{\psi}(H_{\overline{z}})\leq |\delta| D_{\psi}(H_{\overline{z}})$.\\
	The last assertion is a consequence of the second and the third assertion. Indeed, let $\varepsilon >0$ small enough. Suppose that $\delta _1 = \{ e^{i\theta }:\ \theta _1\leq \theta < \theta _2\}$ and $\delta _2 = \{ e^{i\theta }:\ \theta _2\leq \theta < \theta _3\}$ and let $\delta _1(\varepsilon)= \{ e^{i\theta }:\ \theta _1\leq \theta < \theta _2-\varepsilon \}$. By $(2)$, 
$
	\left( H_{\overline{z}}P_\omega M_{\chi_{R_N^{\delta _1(\varepsilon)}}}\right)^\ast\left( H_{\overline{z}}P_\omega M_{\chi_{R_N^{\delta _2}}}\right)\in \displaystyle \cap _{p>0} \cal{S}_p
$. By Corollary \ref{cor1}, we get
$$
D_{\psi \circ \sqrt {.}}\left (\left( H_{\overline{z}}P_\omega M_{\chi_{R_N^{\delta _1} }}\right )^\ast\left( H_{\overline{z}}P_\omega M_{\chi_{R_N^{\delta _2}}}\right ) \right )= D_{\psi\circ \sqrt {.}}\left (\left( H_{\overline{z}}P_\omega M_{\chi_{R_N^{\delta _1}\setminus R_N^{\delta _1(\varepsilon)}}}\right)^\ast\left( H_{\overline{z}}P_\omega M_{\chi_{R_N^{\delta _2}}}\right )\right ).
$$
Applying  Proposition \ref{prop ABC}, we have
\[
\begin{array}{lll}
D_{\psi\circ \sqrt {.}}\left (\left( H_{\overline{z}}P_\omega M_{\chi_{R_N^{\delta _1}\setminus R_N^{\delta _1(\varepsilon)}}}\right)^\ast\left( H_{\overline{z}}P_\omega M_{\chi_{R_N^{\delta _2}}}\right )\right )&\leq & 2 D_{\psi }( H_{\overline{z}}P_\omega M_{\chi_{R_N^{\delta _1}\setminus R_N^{\delta _1(\varepsilon)}}}) D_{\psi }( H_{\overline{z}}P_\omega M_{\chi_{R_N^{\delta _2}}})\\
& \lesssim & \varepsilon D_{\psi }^2( H_{\overline{z}}).\\
\end{array}
\]
Letting $\varepsilon $ to $0$, we obtain $D_{\psi \circ \sqrt {.}}\left (\left( H_{\overline{z}}P_\omega M_{\chi_{R_N^{\delta _1} }}\right )^\ast\left( H_{\overline{z}}P_\omega M_{\chi_{R_N^{\delta _2}}}\right ) \right )=0$, as required.
	\end{proof}

\begin{prop}\label{prop2}
	Let $\omega \in \cW ^\ast$ be a radial weight. Let $\psi \in \cC_p$ such that $0< d_{\psi}(H_{\overline{z}})\leq  D_{\psi}(H_{\overline{z}})<\infty$. 
	 \begin{itemize}
	 \item  Let $\delta\subset\partial\mathbb{D}$ be an arc such that $|\delta|=\frac{2\pi}{N}$, where $N\in \NN ^\ast$. Then
	$$
D_{\psi}(H_{\overline{z}}P_\omega M_{\chi_{R_N^\delta}})=\frac1N D_{\psi}(H_{\overline{z}}) \quad\mbox{and}\quad d_{\psi}(H_{\overline{z}}P_\omega M_{\chi_{R_N^\delta}})=\frac1N d_{\psi}(H_{\overline{z}}) .
	$$ 
	\item Let $g$ be a continuous function on $\overline{\DD}$. We have 
		$$
\| g \| _p^p d_\psi (H_{\overline{z}}) \leq d_\psi (H_{\overline{z}}P_\omega M_{g})\leq D_\psi (H_{\overline{z}}P_\omega M_{g})\leq \| g \| _p^p D_\psi (H_{\overline{z}}).
	$$ 
\end{itemize}
\end{prop}

\begin{proof}
Let
	$$
	R_N=\{z=re^{i\theta}\in\overline{\mathbb{D}}:\quad r>1-1/N \},
	$$
	$$
	R_N^k=\{z=re^{i\theta}\in R_N:\quad \frac{2\pi k}{N}\le\theta<  \frac{2\pi (k+1)}{N}\},\quad k=0,1...,N-1.
	$$
	And let
	$$
	h:=\displaystyle\sum_{k=0}^{N-1}\chi_{R_N^{k}} = 1- \chi_{\DD \setminus R_N} .
	$$
\noindent Hence, by Corollary \ref {cor1}, Lemma \ref{lem1} and Theorem \ref{orthogonal}, we have
\[
\begin{array}{lll}
	D_{\psi}(H_{\overline{z}})&= &D_{\psi}(H_{\overline{z}}P_\omega M_h)\\
	&= & D_{\psi}\left( \displaystyle\sum_{k=0}^{N-1}H_{\overline{z}}P_\omega M_{\chi_{R_N^k}}\right)\\
	&=  & \displaystyle\sum_{k=0}^{N-1}D_{\psi} (H_{\overline{z}}P_\omega M_{\chi_{R_N^k}})\\
	& = & N D_{\psi} (H_{\overline{z}}P_\omega M_{\chi_{R_N^\delta}}). 
\end{array}
\]
Then $D_{\psi}(H_{\overline{z}}P_\omega M_{\chi_{R_N^\delta}})=\frac1N D_{\psi}(H_{\overline{z}})$. Similarly, we obtain $d_{\psi}(H_{\overline{z}}P_\omega M_{\chi_{R_N^\delta}})=\frac1N d_{\psi}(H_{\overline{z}})$.\\
Let $\xi_k$ be the center of the arc $R_N^k\cap\partial\mathbb{D}$ and let 
$$
g_N =\displaystyle\sum_{k=0}^{N-1}g(\xi _k)\chi_{R_N^{k}}.
$$
By Proposition \ref {prop ABC}, we have 
$$
D^{\frac{1}{p+1}}_\psi (H_{\overline{z}}P_\omega M_{g}) \leq D^{\frac{1}{p+1}}_\psi (H_{\overline{z}}P_\omega M_{g_N})+ D^{\frac{1}{p+1}}_\psi (H_{\overline{z}}P_\omega M_{g-g_N})
$$
On one hand, by Lemma \ref{lem1}, we have
$$
D_\psi (H_{\overline{z}}P_\omega M_{g-g_N})\leq \| g-g_N\| ^p_{\infty}D_\psi (H_{\overline{z}}).
$$
On the other hand, by Lemma \ref{lem1}, Theorem \ref{orthogonal} and the first assertion of this proposition, we obtain
\[
\begin{array}{lll}
D_\psi (H_{\overline{z}}P_\omega M_{g_N})& \leq & \displaystyle \sum _{k=0}^{N-1}D_\psi (g(\xi _k)H_{\overline{z}}P_\omega M_{\chi_{R^k_N}})\\
& = & \displaystyle \sum _{k=0}^{N-1}|g(\xi _k) |^pD_\psi (H_{\overline{z}}P_\omega M_{\chi_{R^k_N}})\\
& = & \frac{1}{N}\displaystyle \sum _{k=0}^{N-1}|g(\xi _k)|^pD_\psi (H_{\overline{z}}).\\
\end{array}
\]
Combining these inequalities and letting $N$ going to $\infty$, we obtain 
$$
D_\psi (H_{\overline{z}}P_\omega M_{g})\leq \| g\| _p^pD_\psi (H_{\overline{z}}).
$$
The lower estimates can be obtained by the same arguments.
	\end{proof}


Let $\Phi :R \DD\times R \DD\to\mathbb{C}$ be an analytic function, where $R>1$. Let $A_\Phi:A^2_\omega\to L^2_\omega$ be the operator defined by
\begin{equation}\label{A}
A_\Phi f(z)=\int_{\mathbb{D}}(\overline{z}-\overline{\xi})^2 K(z,\xi)\overline{ \Phi (z,\xi)}f(\xi)\omega(\xi)dA(\xi),\quad z\in\mathbb{D}.
\end{equation}
And let $A$ be the operator given by
\begin{equation}\label{A}
Af(z)=\int_{\mathbb{D}}(\overline{z}-\overline{\xi})^2 K(z,\xi) f(\xi)\omega(\xi)dA(\xi),\quad z\in\mathbb{D}.
\end{equation}
Now, we introduce notations which will be used in the proof of the next lemma. Let $r\in (0,1)$ and let $J_r$ be the embedding operator from $A^2_\omega$ to $ L^2(\omega \chi _{r\DD}dA)$. Note that $T_{r} := J_r^\ast J_r$ is the Toeplitz operator with symbol $\chi _{r\DD}$. Namely, 
$$
T_{r}f = \displaystyle \int _{r\DD}f(\zeta)K(.,\zeta) \omega (\zeta)dA(\zeta).
$$
By \cite {EE},  there exists $C= C(\omega)$ such that 
$$
\displaystyle \sum _{n\geq 1}\lambda ^p_n(T_{r}) \leq \frac{C}{p}\displaystyle \int _0^r\frac{d\rho}{\tau^2_\omega(\rho)}, \quad p\in (0,1).
$$
This implies that 
$$
\lambda _n(T_{r}) \leq \displaystyle \inf _{p\in (0,1)}\left ( \frac{C}{pn}\displaystyle \int _0^r\frac{d\rho}{\tau^2_\omega(\rho)}\right )^\frac 1p = \exp \left ( -\frac {e^{-1}}{C \displaystyle \int _0^r\frac{d\rho}{\tau ^2 _\omega (\rho)}} n\right).
$$
If we suppose that $\tau ^2_\omega (z)\asymp (1-|z|^2)^{2+\beta }\nu ^2(\log (1/1-|z|^2))$, then
$$
\lambda _n(T_r) \leq \exp \left ( -C(\omega)(1-r)^{1+\beta} \nu ^2(\log (1/1-r)) n\right).
$$
In particular we have 
\begin{equation}\label {Jr}
\lambda _n(T_r) = O\left(  \exp \left ( -C(\omega, \varepsilon)(1-r)^{1+\beta-\varepsilon } n\right) \right ), \quad \forall \varepsilon >0.
\end{equation}
 \begin{lem}\label{s_n(APhi)}
	Let $\omega \in \cal {W}^*$ be a radial weight such that $\tau ^2_\omega (z)\asymp  (1-|z|^2)^{2+\beta}\nu ^2(\log  (\frac{1}{1-|z|^2}))$, where $\beta \geq 0$ and $\nu $ is a monotone function such that $\nu (2t) \asymp \nu (t)$. Then, for all $ \varepsilon>0$, we have
	$$
	s_n(A_\Phi)= O\left (\frac {1}{n^{2/p -\varepsilon}}\right ), \quad p = \frac{2(1+\beta)}{2+\beta} .
	$$ 	
\end{lem}
\begin{proof}
First, we prove the result for $A$ which corresponds to $\Phi =1$. Remark that 
$$
A f= H_{\bar{z}^2}f- 2H_{\bar{z}}P_\omega {\bar{z}}f \in A_\omega  ^{2\perp}  .
$$
Then $Af$ is the $L^2_\omega $-minimal solution of $\bar{\partial} u= 2H_{\bar z}f$. Applying Lemma \ref{dbar} twice, we get
\[
\begin{array}{lll}
\| A f\|^2& \lesssim & \displaystyle \int _\DD |2H_{\bar z}f(z)|^2 \tau _\omega ^2(z)\omega (z)dA(z)\\
& \lesssim & \displaystyle \int _{r\DD} |H_{\bar z}f(z)|^2 \omega (z)dA(z)+\tau ^2_\omega (r)\displaystyle \int _{\DD\setminus r\DD} |H_{\bar z}f(z)|^2 \omega (z)dA(z)\\
& \lesssim &  \displaystyle \int _{r\DD} |f(z)|^2 \omega (z)dA(z)+ \displaystyle \int _{r\DD} |P_\omega {\bar z}f(z)|^2 \omega (z)dA(z)+
\tau ^2_\omega (r)\displaystyle \int _{\DD} |f(z)|^2\tau ^2_\omega (z) \omega (z)dA(z)\\
& \lesssim & \| J_r f\|^2+ \| J_rP_\omega \bar{z} f\|^2+ \tau_\omega ^2(r) \langle T_{\tau_\omega ^2}f , f \rangle.
\end{array}
\]
Then we obtain
$$
A^\ast A \lesssim T_r + (P_\omega \bar{z} ) ^\ast T_r  P_\omega \bar{z} + \tau ^2 _\omega(r)T_{\tau_\omega ^2}
$$
This implies that 
\begin{equation}\label {sA}
s^2_{3n}(A) \lesssim \lambda _n( T_r)+\tau_\omega ^2(r)\lambda _n (T_{\tau_\omega ^2}).
\end{equation}
Since $\displaystyle \int ^1\frac{dr}{\tau_\omega ^{2-p -\varepsilon}(r)}< \infty$, by \cite {EE}, $T_{\tau_\omega ^2}\in S_{\frac{p}{2} +\varepsilon}$ for every $\varepsilon >0$. Then, 
\begin{equation}\label {Ttau2}
\lambda _n (T_{\tau_\omega ^2}) = O\left ( \frac{1}{n ^{\frac2p -\varepsilon}}\right ), \quad \mbox{for all}\ \varepsilon >0.
\end{equation}
Combining inequalities (\ref {Jr}), (\ref{sA}) and (\ref {Ttau2}), we obtain 
$$
s^2_{3n}(A)\lesssim \displaystyle \inf _{r\in (0,1)}\left (\exp \left ( -C(\omega, \varepsilon)(1-r)^{1+\beta-\varepsilon } n\right)+ \frac{\tau_\omega ^2(r)}{n ^{\frac2p -\varepsilon}}\right ), \quad \mbox{for all}\ \varepsilon >0.
$$
For a suitable choice of $r$, we obtain the result for $A$. \\
To complete the proof of the Lemma we use the argument given by M. Dostanic in \cite {Dos}. Suppose that $\Phi$ is analytic and bounded by $M$ on $R\DD\times R\DD$. Clearly, we have 
$$\Phi (z, \zeta) = \displaystyle \sum_{k\geq 0}\frac{\Phi ^{(k)}(z,0)}{\fact{k}}\zeta ^k\ \mbox{ and}\  \frac{|\Phi ^{(k)}(z,0)|}{\fact{k}} \leq \frac{M}{R^k}.$$
And since $A_\Phi f=  \displaystyle \sum_{k\geq 0}\frac{\overline {\Phi ^{(k)}(z,0)}}{\fact{k}}AP_\omega \bar{\zeta} ^kf$, we obtain 
$$
s_{(N+2)m}(A_\Phi)\leq  \displaystyle \sum_{k\geq 0}^N \frac{| {\Phi ^{(k)}(z,0)}|}{\fact{k}} s_m(A) + \displaystyle \sum_{k>N} \frac{|\Phi ^{(k)}(z,0)|}{\fact{k}} \|A\|\leq \frac{M}{R-1}\left ( s_m(A) +\|A\| \right ).
$$
Now for $Nm = n$ and $m \sim n^{1-\varepsilon}  $,  we get the result.
	
\end{proof}

\begin{proof}[Proof of Theorem \ref{asymptotics}]
Let $\psi (t)= \frac{1}{\gamma ^p}t^p\nu ^{p/1+\beta}(\log 1/t)$. Since $ s_n(H_{\overline{z}}) \sim \frac{\gamma}{n^{\frac{1}{p}}\nu ^\frac{1}{1+\beta}(\log n)}$, $D_\psi (H_{\overline{z}})= d_\psi (H_{\overline{z}})=1$. Note that 
$$
s_n(H_{\bar{\phi}}) \sim  s_n(H_{\overline{z}}) \| \phi'\| _p \iff D_\psi (H_{\bar{\phi}})= d_\psi (H_{\bar{\phi}})= \| \phi'\| _p^p.
$$
First, suppose that $\phi$ is analytic in a neighborhood of $\overline{\mathbb{D}}$. Then there exist $R>1$ and an analytic bounded function $\Phi$ on $R\DD\times R\DD$ such that
	$$
	\phi(z)-\phi(w)=(z-w)\phi'(w)+(z-w)^2\Phi(z,w),\quad z,w\in R \DD.
	$$
	We have 
	$$
	H_{\overline{\phi}}P_\omega=H_{\overline{z}}P_\omega \bar{\phi'}+A_\Phi P_\omega.
	$$
	By Lemma \ref{s_n(APhi)}
	$$
	s_n(A_\Phi P_\omega)=o\left( s_n(H_{\overline{z}}P_\omega)\right),\quad n\to\infty.
	$$
	So, 
	$$
	D_{\psi}(A_\Phi P_\omega)=0.
	$$
	By Corollary \ref{cor1}, we deduce that 
	$$
	D_{\psi}(H_{\overline{\phi}})=D_{\psi}(H_{\overline{z}}P_\omega \bar{\phi'})\quad\mbox{and}\quad d_{\psi}(H_{\overline{\phi}})=d_{\psi}(H_{\overline{z}}P_\omega \bar{\phi'}).
	$$
	We obtain, by Proposition \ref{prop2}, that
	$$
	d_{\psi}( H_{\overline{z}})\|\phi'\|^{p}_{p}\leq d_{\psi}(H_{\overline{\phi}})\leq D_{\psi}(H_{\overline{\phi}})\leq D_{\psi}( H_{\overline{z}})\|\phi'\|^{p}_{p}.
	$$
Since  $D_{\psi}( H_{\overline{z}})=d_{\psi}( H_{\overline{z}})=1$, we obtain 
$$
d_{\psi}(H_{\overline{\phi}})=D_{\psi}(H_{\overline{\phi}})=\|\phi'\|^{p}_{p}.
$$
Now, suppose that $\phi ' \in H^p$. By Theorem \ref{Hp}, we have
$$
s_n (H_{\bar{\phi} -\bar{\phi}_r})\leq C\| \phi '-\phi _r'\|_p s_n(H_{\overline{z}}), \quad n\geq 1.
$$
This implies that $D_\psi (	H_{\bar{\phi} -\bar{\phi}_r}) \leq C^p\| \phi '-\phi _r'\|_p^pD_\psi (H_{\overline{z}})$. Then we have
\[
\begin{array}{lll}
|D_{\psi}(H_{\overline{\phi}})^{\frac{1}{p+1}}-\left ( D_{\psi}( H_{\overline{z}})\|\phi_r'\|^{p}_{p}\right )^{\frac{1}{p+1}}|& = & |D_{\psi}(H_{\overline{\phi}})^{\frac{1}{p+1}}-D_{\psi}(H_{\overline{\phi_r}})^{\frac{1}{p+1}}|\\
&\leq & D_{\psi}(H_{\overline{\phi}-\overline{\phi_r} })^{\frac{1}{p+1}}\\
&\leq & \left ( C^p\| \phi '-\phi_r'\| _p ^p D_\psi (H_{\overline {z}})\right ) ^{\frac{1}{p+1}}.\\
\end{array}
\]
When $r\to 1^-$, we obtain  $D_{\psi}(H_{\overline{\phi}})=D_{\psi}( H_{\overline{z}})\|\phi'\|^{p}_{p}=\|\phi'\|^{p}_{p}$. With the same arguments we have $d_{\psi}(H_{\overline{\phi}})=d_{\psi}( H_{\overline{z}})\|\phi'\|^{p}_{p}$. The proof is complete.
\end{proof}
Similarly, one can prove the following result which corresponds to the situation $\beta=0$  and $\nu \asymp 1$ in Theorem \ref{asymptotics}.

\begin{thm}\label{beta=0}
Let $\omega \in \cal {W}^*$ be a radial weight such that $\tau ^2_\omega (z)\asymp  (1-|z|^2)^{2}$.  Suppose that 
$$
s_n(H_{\overline{z}}) \sim \frac{\gamma}{n},
$$
where $\gamma > 0$. Then, for $\phi ' \in H^{1}$, we have
$$s_n(H_{\overline{\phi}}) \sim  \frac{\gamma}{n}\| \phi '\|_1.$$
\end{thm}
\begin{proof}[Proof of the second assertion of Theorem \ref{main result 2}]
It suffices to combine Theorem \ref{beta=0} and Theorem \ref{asymptotics} with $\nu =1$.
\end{proof}
\section{An example}\label{Example}
For a radial weight $\omega$, the sequence of the singular values of the Hankel operator $H_{\overline{z}}$ on $A^2_\omega$ is the decreasing rearrangement of the sequence
$$\left( \left ( \frac{|| z^{n+1} ||^2}{|| z^n ||^2}	- \frac{|| z^n ||^2}{|| z^{n-1} ||^2} \right ) ^{1/2} \right )_n.$$
Recall that for the standard case $\omega _\alpha (z)= (\alpha +1)(1-|z|^2)^\alpha$, we have $s_n(H_{\overline{z}})\sim \frac{\sqrt{\alpha +1}}{n}$. In this section, we consider the weight 
$$ \omega(z) = \exp \left ( - \frac{\alpha}{ (\log \frac{1}{|z|^2} )^{\beta} }\right ),\quad \alpha , \ \beta > 0.$$
We have
\begin{align*}
|| z^n ||^2 &= \int_\mathbb{D} |z|^{2n}   \exp \left ( - \frac{\alpha}{ (\log \frac{1}{|z|^2} )^{\beta} }\right )dA(z) \\
				&= \int_0^1 r^{2n} \exp \left (-\frac{\alpha}{ (\log \frac{1}{r^2} )^{\beta} }\right )2rdr \\
				&= 	\int_0^{+\infty} \exp\left(-(n+1)x - \frac{\alpha}{x^\beta}\right)dx. \hspace{3cm} 
\end{align*}
Let $x_n := \left ( \frac{\alpha \beta}{n+1} \right )^{1/1+\beta }$ be the minimum of the function $(n+1)x+\frac{\alpha}{x^\beta}$. After the change of variable $u = \frac{x-x_n}{x_n}$, we get 

$$	|| z^n ||^2 = x_n \exp \left (-(n+1)x_n -\frac{ \alpha}{x_n^\beta} \right ) \int_{-1}^{+\infty} \exp \left (-\frac{\alpha}{x_n^\beta}h(u)\right ) du, 	$$ 
where $h(u) = \beta u + \frac{1}{(1+u)^\beta} -1$. \\
In the sequel, we will use Laplace method \cite{Erd} to get an expansion of $ || z^n ||^2$. By Laplace Theorem, we have 
$$
H(t) = \sqrt{\frac{t h''(0)}{2\pi}} \int_{-1}^{+\infty} \exp \left (-t h(u)\right ) du \sim 1,\quad t\to +\infty.
$$
Let $\eta>0$. We have 
$$
\sqrt{\frac{t h''(0)}{2\pi}} \int_{[-1,-\eta]\cup [\eta, +\infty[} \exp \left (-t h(u)\right ) du = O(e^{-c(\eta)t}),\quad t\to +\infty.
$$
Let 
\begin{align*}
H^+_\eta (t)& = \sqrt{\frac{t h''(0)}{2\pi}} \int_{0< u <\eta} \exp \left (-t h(u)\right )du,\\  
H^-_\eta (t)& = \sqrt{\frac{t h''(0)}{2\pi}} \int_{-\eta< u <0} \exp \left (-t h(u)\right )du.
\end{align*} 
Using the change of variable $v=h(u)$, when $\eta$ is small enough, one can write
$$
H^+_\eta (t)= \displaystyle \sum _{j=0}^{2N}\frac{c_j}{t^{j/2}}+o\left (\frac{1}{t^N}\right ),
$$
where $c_0, .., c_{2N} \in \RR$. The same is also true for $H^-_\eta$. Then
$$
H(t) = \displaystyle \sum _{j=0}^{2N}\frac{d_j}{t^{j/2}}+o\left (\frac{1}{t^N}\right ),
$$
where $d_0=1$ and $d_1,..,d_{2N} \in \RR$.
Let $A_n =  \sqrt{\frac {2\pi} {h''(0)}}\  x_n^{1+\frac{\beta}{2}} \exp \left (-(n+1)x_n - \frac{\alpha}{x_n^\beta } \right )$. We have 
$$
|| z^n ||^2= A_n H\left (\frac{\alpha}{x^{\beta}_n}\right ).
$$
By a direct calculation, we get
$$	
\frac{|| z^n ||^2}{|| z^{n-1} ||^2} = 1  - \frac{a}{n^{1/1+\beta}} + \frac{b}{n}+ \frac{c}{n^{1+1/1+\beta}}+ o(1/n^{1 + \frac{1}{\beta + 1}}),	
$$
where $a= (\alpha \beta )^{1/1+\beta}$. Then, we have 
$$
\frac{|| z^{n+1} ||^2}{|| z^{n} ||^2} -\frac{|| z^n ||^2}{|| z^{n-1} ||^2}  \sim \frac{\gamma ^2}{n^{\frac{\beta +2}{\beta+1} } }, 
$$
where $\gamma = \sqrt{ \frac{(\alpha \beta)^{1/1+\beta}}{1+\beta}}$. Finally, we obtain 
$$s_n(H_{\overline {z}})\sim \frac{\gamma }{n^{\frac{\beta +2}{2(\beta+1)}}}.$$



\section*{Appendix: Asymptotic Orthogonality}\label {Annexe}


Let $T$ be a compact operator between two complex Hilbert spaces. The decreasing sequence of singular values of $T$ will be denoted by $(s_n(T))_n$.The counting function of  the singular values of $T$ is denoted by
$$
n(s,T)=\# \{ n: \ s_n(T) \geq s\} ,\quad s>0.
$$	
Recall that $\cC _p$ denotes the class of  increasing continuous function $\psi:(0,+\infty)\to (0,+\infty)$ satisfying 
$$
\psi(\alpha t)\sim \alpha^p \psi(t),\quad(t\to 0^+),
$$
As before, $D _\psi (T), d _\psi (T)$ are given by
$$
D _\psi (T):=\displaystyle\limsup_{s\to 0^+}\psi(s)n(s,T)\quad\mbox{and}\quad d _\psi (T):=\displaystyle\liminf_{s\to 0^+}\psi(s)n(s,T).
$$
The goal of this Annex is to extend the results obtained for $\psi (t) = t^p$, \cite{BS, Pus}, to the class $\cC _p$. We give here the proof for completeness.
\begin{prop}\label{prop ABC}
	Let $T$ and $V$ be compact operators and $\lambda\in\mathbb{C}$. We have
	\begin{enumerate}
		\item $D_\psi(\lambda T)=|\lambda|^pD_\psi(T)$ and $d_\psi(\lambda T)=|\lambda|^pd_\psi(T)$.
		\item $D_\psi(T+V)^{\frac{1}{p+1}}\leq D_\psi(T)^{\frac{1}{p+1}}+ D_\psi(V)^{\frac{1}{p+1}}$.
		\item $d_\psi(T+V)^{\frac{1}{p+1}}\leq d_\psi(T)^{\frac{1}{p+1}}+ D_\psi(V)^{\frac{1}{p+1}}$.
		\item $D_\psi(TV)\leq 2D_{\psi\circ t^2}(T) D_{\psi\circ t^2}(V)$.
		
	\end{enumerate}
\end{prop}
\begin{proof}
	$1.$ For $\lambda\neq 0$, we have
		\begin{align*}
		D_\psi(\lambda T)=\displaystyle\limsup_{s\to 0^+}\psi(s)n(s,\lambda T)&=\displaystyle\limsup_{s\to 0^+}\psi(s)n(s/|\lambda|,T)\\
		&=\displaystyle\limsup_{s\to 0^+}\psi(|\lambda|s)n(s,T)\\
		&=|\lambda|^p \displaystyle\limsup_{s\to 0^+}\psi(s)n(s,T)\\
		&=|\lambda|^pD_\psi(T).
		\end{align*}
		Similarly, we have $d_\psi(\lambda T)=|\lambda|^pd_\psi(T)$, $\forall \lambda\in\mathbb{C}$.\\
		$2.$ Let $x\in(0,1)$. For all $s>0$, we have
		\begin{align*}
		n(s,T+V)=n(xs+(1-x)s,T+V)&\leq n(xs,T)+n((1-x)s,V)\\
		&= n\left( s,\frac1x T\right) +n\left( s,\frac{1}{1-x}V\right) .
		\end{align*}
		Hence, $D_\psi(T+V)\leq D_\psi\left( \frac1x T\right) + D_\psi\left( \frac{1}{1-x} V\right) $. By the first assertion we obtain
		$$
		D_\psi(T+V)\leq x^{-p} D_\psi( T)+(1-x)^{-p}D_\psi( V).
		$$
		It follows that
		$$
		D_\psi(T+V)\leq \displaystyle\min_{x\in(0,1)} \left\lbrace x^{-p} D_\psi( T)+(1-x)^{-p}D_\psi( V)\right\rbrace,
		$$
		and $2.$ is obtained. \\
		By the same way we obtain also the third assertion.\\
		$4.$ For all $\alpha,s>0$, we have
		\begin{align*}
		n(s,TV)=n(\alpha\sqrt{s}\;\alpha^{-1}\sqrt{s},TV)&\leq n(\alpha\sqrt{s},T)+n(\alpha^{-1}\sqrt{s},V)\\
		&= n(\sqrt{s},\frac1\alpha T)+n(\sqrt{s},\alpha V).
		\end{align*}
		Hence
		\begin{align*}
		\psi(s)n(s,TV)
		&\leq \psi(s) n(\sqrt{s},\frac1\alpha T)+\psi(s)n(\sqrt{s},\alpha V)\\
		&= \tilde{\psi}(\sqrt{s}) n(\sqrt{s},\frac1\alpha T)+\tilde{\psi}(\sqrt{s})n(\sqrt{s},\alpha V),
		\end{align*}
		with $\tilde{\psi}=\psi (t^2)$. Then
		\begin{align*}
		D_\psi(TV)&\leq D_{\tilde{\psi}}\left( \frac1\alpha T\right) + D_{\tilde{\psi}}(\alpha V)\\
		&=\alpha^{-2p}D_{\tilde{\psi}}( T)+ \alpha^{2p} D_{\tilde{\psi}}( V).
		\end{align*}
		It follows that
	\begin{align*}
		D_\psi(TV)&\leq \displaystyle\min_{\alpha>0} \left\lbrace \alpha^{-2p}D_{\tilde{\psi}}( T)+ \alpha^{2p} D_{\tilde{\psi}}( V)\right\rbrace\\
		&\leq 2 D_{\tilde{\psi}}( T) D_{\tilde{\psi}}( V).
		\end{align*}	
\end{proof}
Let us denote
$$
\Sigma_\psi:=\left\lbrace T\in S_\infty:\quad\psi(s)n(s,T)=O(1) \right\rbrace 
$$
and
$$
\quad\mbox{and}\quad \Sigma^0_\psi:=\left\lbrace T\in S_\infty:\quad\psi(s)n(s,T)=o(1) \right\rbrace.
$$ 
The second and the third assertions of the Proposition \ref{prop ABC} imply the following corollary.
\begin{cor}\label{cor1}
	Let $T$ and $V$ be compact operators. Suppose that $V\in \Sigma^0_\psi$, then
	$$
	D_\psi(T+V)=D_\psi(T)\quad\mbox{and}\quad d_\psi(T+V)=d_\psi(T).
	$$
\end{cor}

Once we have that, we can state the Theorem of A. Pushnitski \big[\cite {Pus}, Theorem 2.2\big] in the following form 
\begin{thm}\label{orthogonal}
Let $T_1,T_2,...,T_n$ be a compact operators such that
$$
T_i^\ast T_j\in \Sigma^0_{\psi\circ\sqrt{.}} \quad\mbox{and}\quad T_i T_j^\ast\in \Sigma^0_{\psi\circ\sqrt{.}},\quad\forall i\neq j.
$$	
Then
$$
D_\psi\left( \displaystyle\sum_{k=1}^{n}T_k\right) =\displaystyle\limsup_{s\to 0^+}\psi(s)\displaystyle\sum_{k=1}^{n}n(s,T_k)\quad\mbox{and}\quad d_\psi\left( \displaystyle\sum_{k=1}^{n}T_k\right) =\displaystyle\liminf_{s\to 0^+}\psi(s)\displaystyle\sum_{k=1}^{n}n(s,T_k).
$$
In particular, if for all $s>0$, we have $n(s,T_k)=n(s,T_1)$, $k=1,2,...,n$, then
$$
D_\psi\left( \displaystyle\sum_{k=1}^{n}T_k\right)=nD_\psi\left(T_1\right)\quad\mbox{and}\quad d_\psi\left( \displaystyle\sum_{k=1}^{n}T_k\right)=nd_\psi\left(T_1\right).
$$ 
\end{thm}
\begin{proof}
	Let $T=\displaystyle\sum_{k=1}^{n}T_k$. Since
	$$
	TT^\ast=\displaystyle\sum_{i=1}^{n}T_iT_i^\ast+ \displaystyle\sum_{i,j=1 ,i\neq j}^{n}T_iT_j^\ast,
	$$
	and $T_i T_j^\ast\in \Sigma^0_{\psi\circ\sqrt{.}}$, for $i\neq j$, we have 
	\begin{align*}
		D_\psi(T)=D_{\psi\circ\sqrt{.}}(TT^\ast)
		&=D_{\psi\circ\sqrt{.}}\left(\displaystyle\sum_{i=1}^{n}T_iT_i^\ast \right)\\
		&=D_{\psi\circ\sqrt{.}}\left( (JA)(JA)^\ast\right)\\
		&= D_{\psi\circ\sqrt{.}}\left( (JA)^\ast(JA)\right).
	\end{align*}
	Where
	$$
	\begin{array}{ccccc}
	J & : & H^n & \longrightarrow & H \\
	& & (f_1,...,f_n) & \longrightarrow &  \displaystyle\sum_{i=1}^{n}f_i
	\end{array} \qquad\mbox{and}\qquad
	\begin{array}{ccccc}
	A & : & H^n & \longrightarrow & H^n \\
	& & (f_1,...,f_n) & \longrightarrow &  (T_1f_1,T_2f_2,...,T_nf_n)
	\end{array}
	 $$
	 The matrix of $(JA)^\ast(JA)$ is
	 \begin{align*}
	 \begin{pmatrix}
	 T_1^\ast T_1 & T_1^\ast T_2 & \dots & T_1^\ast T_n\\
	 T_2^\ast T_1 & T_2^\ast T_2 & \dots & T_2^\ast T_n\\
	 \vdots& \vdots & {\ddots}  & \vdots\\
	 T_n^\ast T_1 & T_n^\ast T_2 & \dots & T_n^\ast T_n
	 \end{pmatrix}
	 =
	 \begin{pmatrix}
	 T_1^\ast T_1 &  & &0\\
	  & T_2^\ast T_2 &  & \\
	 &  & {\ddots}  & \\
	 0 &  &  & T_n^\ast T_n
	 \end{pmatrix}
	 +
	 \begin{pmatrix}
	 O & T_1^\ast T_2 & \dots & T_1^\ast T_n\\
	 T_2^\ast T_1 & 0 & \dots & T_2^\ast T_n\\
	 \vdots& \vdots & {\ddots}  & \vdots\\
	 T_n^\ast T_1 & T_n^\ast T_2 & \dots & 0
	 \end{pmatrix}
	 =:T'+T''.
	 \end{align*}
	 Since $T_i^\ast T_j\in \Sigma^0_{\psi\circ\sqrt{.}}$, for $i\neq j$, we deduce that 
	 \begin{align*}
	 	D_\psi(T)=D_{\psi\circ\sqrt{.}}(T'+T'') 
	 	&=D_{\psi\circ\sqrt{.}}(T')\\
	 	&=D_{\psi\circ\sqrt{.}}(T_1^\ast T_1,T_2^\ast T_2,...,T_n^\ast T_n).
	 \end{align*}
	 Let now $T_0$ the operator
	 $$
	 T_0:=\begin{pmatrix}
	 T_1 &  & &0\\
	 & T_2 &  & \\
	 &  & {\ddots}  & \\
	 0 &  &  & T_n
	 \end{pmatrix}
	 $$
	 We have
	 \begin{align*}
	 D_\psi(T)=D_{\psi\circ\sqrt{.}}(T_1^\ast T_1,T_2^\ast T_2,...,T_n^\ast T_n)
	 &=D_{\psi\circ\sqrt{.}}(T_0^\ast T_0)\\
	 &=D_{\psi}(T_0)\\
	 &=\displaystyle\limsup_{s\to 0^+}\psi(s)\displaystyle\sum_{k=1}^{n}n(s,T_k).
	 \end{align*}
	 In the same way we prove that
	 $$
	 d_\psi\left( \displaystyle\sum_{k=1}^{n}T_k\right) =\displaystyle\liminf_{s\to 0^+}\psi(s)\displaystyle\sum_{k=1}^{n}n(s,T_k).
	 $$
\end{proof}


\bibliographystyle{amsplain}	
\bibliography{biblHankel}

	\end{document}